\documentclass[11pt]{amsart}
\usepackage[english]{babel}
\usepackage[T1]{fontenc}
\usepackage[ansinew]{inputenc}
\usepackage{amsmath}
\usepackage{amsfonts}
\usepackage{ mathrsfs }
\usepackage{amssymb}
\usepackage{textcomp}
\usepackage{enumitem}
\usepackage[hidelinks]{hyperref}
\usepackage[arrow, matrix, curve]{xy}
\usepackage{comment}
\usepackage{color}

\numberwithin{paragraph}{section}
\numberwithin{equation}{section}

\newtheorem{satz}{Theorem}[section]

\newtheorem{lem}[satz]{Lemma}

\newtheorem{prop}[satz]{Proposition}
\newtheorem{Prop}[satz]{Proposition}
\newtheorem{kor}[satz]{Corollary}
\theoremstyle{definition}
\newtheorem{defn}[satz]{Definition}
\newtheorem{bem}[satz]{Remark}

\newtheorem{Const}[satz]{Construction}

\newtheorem{theorem}{Theorem}[]
\newtheorem{thm}[theorem]{Theorem}

\newtheorem*{ack}{Acknowledgments}

\newcommand{\Z}{\mathbb{Z}}
\newcommand{\TT}{\mathbb{T}}
\newcommand{\Q}{\mathbb{Q}}
\newcommand{\R}{\mathbb{R}}
\newcommand{\C}{\mathbb{C}}
\newcommand{\T}{\mathbb{T}}

\newcommand{\G}{\mathbb{G}}

\newcommand{\Linear}{\mathbb{L}}

\newcommand{\vedge}{\land}
\newcommand{\del}{\partial}
\newcommand{\Xan}{X^{\an}}
\newcommand{\inj}{\hookrightarrow}

	\DeclareMathOperator{\an}{an}
	\DeclareMathOperator{\PD}{PD}

	\DeclareMathOperator{\trop}{trop}

	\DeclareMathOperator{\Hom}{Hom}
	
	\DeclareMathOperator{\Spec}{Spec}

	\DeclareMathOperator{\Trop}{Trop}

	\DeclareMathOperator{\sing}{sing}
	\DeclareMathOperator{\supp}{supp}

	\DeclareMathOperator{\id}{id}

	\DeclareMathOperator{\val}{val}
	\DeclareMathOperator{\CH}{CH}
	\DeclareMathOperator{\HH}{H}

\DeclareMathOperator{\AS}{\mathcal{A}}

\DeclareMathOperator{\CS}{\mathcal{C}}

\DeclareMathOperator{\GS}{\mathcal{G}}

\DeclareMathOperator{\LS}{\mathcal{L}}

\DeclareMathOperator{\OS}{\mathcal{O}}

\DeclareMathOperator{\XS}{\mathcal{X}}

\DeclareMathOperator{\A}{\mathbb{A}}

\def\quotient#1#2{\raise0.75ex\hbox{$\,#1$}\big/\lower0.75ex\hbox{$#2\,$}}

\title[Tropical Dolbeault cohomology of Mumford curves]{Poincar\'e duality for the tropical Dolbeault cohomology of non-archimedean Mumford curves}

\author[P.~Jell]{Philipp Jell}
\address{P. Jell, Mathematik, Universit{\"a}t 
Regensburg, 93040 Regensburg, Germany}
\email{philipp.jell@mathematik.uni-regensburg.de}

\author[V.~Wanner]{Veronika Wanner}
\address{V. Wanner, Mathematik, Universit{\"a}t 
Regensburg, 93040 Regensburg, Germany}
\email{veronika.wanner@mathematik.uni-regensburg.de}

\thanks{Both authors were supported by the collaborative research center SFB 1085 "Higher Invariants" funded by the Deutsche Forschungsgemeinschaft.
 }

\begin{document}
\begin{abstract}
We calculate the tropical Dolbeault cohomology for the analytifications of $\mathbb{P}^{1}$ and Mumford curves 
over non-archimedean fields. 
We show that the cohomology satisfies Poincar\'e duality and 
behaves analogously to the cohomology of curves over the complex numbers. 
Further, we give a complete calculation of the dimension of the cohomology on a basis of the topology.

\bigskip

\noindent
MSC: Primary 32P05; Secondary  14T05, 14G22, 14G40

\bigskip

\noindent
Keywords: Non-archimedean geometry, Berkovich spaces, Non-archimedean curves, Mumford curves, Smooth differential forms on Berkovich spaces, Tropicalizations,  Poincar\'e duality
\end{abstract}

\maketitle 
\tableofcontents

\section{Introduction}

In their preprint \cite{CLD}, Chambert--Loir and Ducros
introduced bigraded real-valued differential forms on Berkovich analytic spaces. 
These forms are analogues of differential forms on complex manifolds. 
Their construction is based on the definition of $(p, q)$-superforms on open subsets
of $\R^r$ together with linear differential operators $d'$ and $d''$ by Lagerberg \cite{Lagerberg}. 
For a Berkovich analytic space $X$, one obtains a double complex $(\AS_X^{\bullet, \bullet}, d', d'')$ of fine sheaves of real vector spaces.
It arose the natural question of whether the $d''$-cohomology $\HH^{p,q} := \HH^q(\AS^{p, \bullet}, d'')$ defined by these forms, 
and its dimension $h^{p,q} := \dim_\R \HH^{p,q}$, 
behave in a similar way as in the complex case. 
A Poincar\'e lemma was established in \cite{Jell}. 
This implies that the cohomology of these forms encodes the singular cohomology of $X$.
Namely, we have $\HH^{q}_{\sing}(X, \R) = \HH^{0,q}(X)$.

Let $K$ be an algebraically closed, complete, non-archimedean field and $X$ a smooth proper variety over $K$ with Berkovich analytification $\Xan$. 
As a consequence of the above, 
we have $\HH^{0,q}(\Xan) = 0$ for all $q > 0$ if $X$ has good reduction.
In particular, we do not have $h^{0,1}(\Xan) = g$ for all smooth proper curves $X$ of genus $g$. 
To our knowledge, the only calculation of $h^{p,q}(\Xan)$ for $p >0$ is that
if $X$ has good reduction, we have $h^{p,0}(\Xan) = 0$ for $p =1$ or $p = \dim(X)$ \cite[Proposition 3.4.11]{JellThesis} and 
for other $p$ under further assumptions \cite[Theorem 1.1]{Liu2}.
At this point, let us mention that the cohomology of differential forms introduced by Chambert--Loir and Ducros was also studied by Yifeng Liu, 
who amongst other things defines a cycle class map $\CH^p(X) \rightarrow \HH^{p,p}(\Xan)$ 
and shows finite dimensionality of $\HH^{1,1}(\Xan)$ for $X$ proper and smooth over $K = \C_p$  \cite[Theorem 1.8]{Liu}.

In the case that the analytic space is the analytification of an algebraic variety, 
there have been slight modifications of the definitions by Chambert-Loir and Ducros by 
Gubler \cite{Gubler} and the first author \cite{JellThesis}.
Both of these approaches lead canonically to the same forms as the approach by Chambert--Loir and Ducros if the absolute value of $K$ is non-trivial.

In this present work, we give a complete calculation of $h^{p,q}(\Xan)$ for $\mathbb{P}^{1}$ and Mumford curves, 
i.e.~smooth projective curves $X$ such that the special fibre of a semistable model has only rational irreducible components 
(cf.~Section \ref{Section 2} for some general properties of these curves).
We indeed find as in the complex case the following dimensions. 
Note that $h^{p,q}(X^{\an})=0$ for every algebraic curve $X$ if $p>1$ or $q>1$.
\begin{thm}\label{Theorem Einl. 1}
	Let either $X$ be $\mathbb{P}^1_K$ or $K$ be non-trivially valued and $X$ a Mumford curve over $K$.
	We denote by $g$ the genus of $X$ and let  $p,q\in\{0,1\}$. 
	Then we have
	\begin{align*}
	h^{p,q}(\Xan) = 
	\begin{cases} 
	1 \text{ if } p = q, \\
	g \text{ else}.
	\end{cases}
	\end{align*}
\end{thm}
The key statement to obtain this result is Poincar\'e duality for certain open subsets $V$ of $\Xan$ for a smooth algebraic curve $X$. 
We write $\HH^{p,q}_c(V) := \HH^q(\AS^{p, \bullet}_c(V), d'')$, where $\AS^{p,q}_c(V)$ are the sections of $\AS^{p,q}(V)$ with compact support.
For the definition of points of type $2$ and their genus see  Definition \ref{Def. Genus}.

\begin{thm}[Poincar\'e Duality]\label{Theorem Einl PD}
Let $K$ be non-trivially valued, $X$  a smooth curve over $K$ and 
$V \subset \Xan$ an open subset such that all points of type $2$ in $V$ have genus $0$. 
Then  
\begin{align*}
\HH^{p,q}(V) \rightarrow \HH^{1-p,1-q}_c(V)^*,~[\alpha]\mapsto ([\beta]\mapsto \int\nolimits_{V}\alpha\wedge \beta)
\end{align*}
is an isomorphism for all $p,q$.

If $K$ is trivially valued, then the statement is still true for any open subset $V$ of $\mathbb{P}^{1, \an}_{K}$.
\end{thm}

If $K$ is non-trivially valued, for a smooth curve $X$ a basis of the topology of $\Xan$ is given by so called 
strictly simple open subset (for a definition see \ref{Def. simple}). 
Theorem \ref{Theorem Einl PD} enables us to calculate $h^{p,q}$ and $h^{p,q}_c$ for strictly 
simple open open subsets which do not contain type $2$ points of positive genus. 
By the boundary of such a strictly simple open subset we mean the topological boundary in $\Xan$. 
In particular, we show in Corollary \ref{the boundary is finite} that this boundary is finite.

\begin{thm}\label{Theorem Einl 2} 
Let $K$ be non-trivially valued, $X$ be a smooth projective  curve over $K$ and $p,q\in\{0,1\}$. 
Let $V$ be a strictly simple open subset of $\Xan$ such that all type $2$ points in $V$ have genus $0$ and
denote by $k := \# \del V$ the finite number of boundary points. 
Then we have 
	\begin{align*}
	h^{p,q}(V) = 
	\begin{cases}
	1 &\text{ if } (p,q) = (0,0) \\
	k - 1 &\text{ if } (p,q) = (1,0) \\
	0 &\text{ if } q \neq 0
	\end{cases}
	\text{ and }
	h^{p,q}_c(V) = 
	\begin{cases}
	1 &\text{ if } (p,q) = (1,1) \\
	k - 1 &\text{ if } (p,q) = (0,1) \\
	0 &\text{ if } q \neq 1.
	\end{cases}
	\end{align*}
\end{thm}
Note additionally that for any smooth algebraic curve $X$ of genus $g$, 
the space $\Xan$ contains at most $g$ points of type $2$ with positive genus \cite[Remark 4.18]{BPR2}.
Thus this theorem describes the cohomology locally at all but finitely many points.
Further, if $X$ is a Mumford curve, then $\Xan$ contains no type $2$ points of positive genus 
(Theorem \ref{Satz Mumford curve} and Proposition \ref{Prop Thm Mumford curve}). 
Thus Theorem \ref{Theorem Einl 2} describes the cohomology on a basis of the topology.

We want to give a quick overview of the techniques we use. 
Differential forms on Berkovich analytic spaces are defined using superforms in the sense of Lagerberg \cite{Lagerberg} on tropical varieties. 
For a tropical variety $Y$ one obtains a double complex $(\AS^{\bullet,\bullet}_Y, d', d'')$ of sheaves on $Y$. 
These forms are then pulled back along the tropicalization map.
For an open subset $V$ of $\A^{1, \an}$ one a priori needs to consider all tropicalizations to determine the cohomology 
$\HH^{p,q}(V)$ resp.~$\HH^{p,q}_c(V)$.
Our key idea is to use the invariance of $\HH^{p,q}(Y)$ resp.~$\HH^{p,q}_c(Y)$ 
along tropical modifications (cf.~Proposition \ref{prop:closedmodification}) 
to show that,
for a large class of open subsets $V$ of $\A^{1, \an}$, the cohomology $\HH^{p,q}(V)$ resp.~$\HH^{p,q}_c(V)$
can be calculated using only one single tropicalization (cf.~Theorem \ref{thm Cohom. lineare trop.}).

Since we can further show that this tropicalization can be taken such that the resulting tropical curve is smooth, we can use the fact
that open subsets of a smooth tropical curve satisfy Poincar\'e duality to prove that $V$ does as well and 
then proceed to prove Theorems \ref{Theorem Einl. 1} and \ref{Theorem Einl 2}.

In this sense, the present paper can be seen as an application of the tropical statements of \cite[Chapter 4]{JSS}
to the Berkovich analytic setting. 
To make this work we use $\A$-tropical charts, i.e.~tropicalizations of closed embeddings 
$\varphi\colon U\to \A^{r}$ for an affine open subset $U$ of $X$, as in \cite{JellThesis}.

If one wishes to apply our techniques to more general curves resp.~varieties, one needs to have a better understanding
of when refinements between tropical charts induce tropical modifications and when tropicalizations are smooth.

In Section \ref{Section 2},  we recall all relevant notions and statements from previous work, 
including (smooth) tropical curves, superforms, forms on Berkovich spaces, Poincar\'e duality, tropical modifications and Mumford curves.

In Section \ref{section modifications}, we define linear $\A$-tropical charts $(V,\varphi)$. 
These are given by closed embeddings $\varphi\colon \A^1 \to \A^{r}$
such that the corresponding algebra homomorphism $\varphi^\sharp$ is given by linear polynomials and 
$V=\varphi_{\trop}^{-1}(\Omega)$ for an open subset $\Omega$ of $\Trop_\varphi(\A^1)$ 
(for the definition of $\Trop_\varphi(\A^1)$ see Definition \ref{Def. transition map}).
We will show that refinements of linear $\A$-tropical charts (cf.~Definition \ref{defn:GATcharts}) induce tropical modifications on the tropical level 
(cf.~Theorem \ref{theorem refinement modification}) and that $\Trop_\varphi(U)$ is a smooth tropical curve 
(cf.~Theorem \ref{theorem tropical manifold}). 
Combining this result with the fact that $\HH^{p,q}$ and $\HH^{p,q}_c$ are invariant under pullback along tropical modifications 
leads to the main result of this section: 
$\HH^{p,q}_{c}(V)$ is isomorphic to $\HH^{p,q}_{c}(\varphi_{\trop}(V))$ for any linear $\A$-tropical chart $(V,\varphi)$ 
where $V$ is a standard open subset of $\A^{1,\an}$. 
A definition of standard open subsets of $\A^{1,\an}$ is given in Definition \ref{Def. Trop. Grundlagen}.

In Section \ref{Section 4}, we prove Poincar\'e duality for open subsets $V$ of $\Xan$, 
where $X$ is a smooth algebraic curve and  $V$ does not contain any point of type $2$ which has positive genus. 
For this we use our considerations from Section \ref{section modifications} and 
the fact that these sets $V$ are locally isomorphic to open subsets of  $\A^{1,\an}$ (cf.~Proposition \ref{Prop Thm Mumford curve}).
For the proof of Poincar\'e duality, we first show that one can prove Poincar\'e duality locally.  
Thus, it follows from the fact that $\HH^{p,q}_{c}(V)$ is isomorphic to 
$\HH^{p,q}_{c}(\varphi_{\trop}(V))$ and  $\varphi_{\trop}(V)$ having Poincar\'e duality for suitable open subsets $V$, which is known from \cite{JSS}. 

In Section \ref{Section 5}, we apply Poincar\'e duality to prove Theorem \ref{Theorem Einl. 1} and Theorem \ref{Theorem Einl 2}.

\subsection*{Recent developments}

Since the first version of this paper was posted, two further papers studying tropical Dolbeault 
cohomology of curves have been posted. 
Y. Liu showed that for smooth projective curves $X$ there exists an isomorphism $N \colon \HH^{1,0}(\Xan) \to \HH^{0,1}(\Xan)$,
when the following assumptions are satisfied: $X$ is the base change to the completed algebraic closure 
of a smooth projective curve $X_0$ over a field $K_0$,
where $K_0$ is either a finite extension of $\Q_p$ or $k((t))$ for a finite field $k$. 
Further, $X$ has to admit a strictly semistable model over the valuation ring of $K_0$ \cite[Theorem 1.1]{Liu2}.
Also the first author showed that if the residue field of $K$ is the algebraic closure of a finite field,  
all smooth projective curves over $K$ satisfy Poincar\'e duality \cite[Corollary 4.3]{JellDuality}.
He also proves some local results, using Theorem \ref{Theorem Einl PD} from this paper.

\ack{The authors would like to thank Walter Gubler and Klaus K\"unnemann for very carefully reading drafts of this work and 
Vladimir Berkovich for explaining us that Proposition \ref{Prop Thm Mumford curve} follows from his work.
The authors are also grateful to the anonymous referee for their very detailed report, which
helped a lot to improve the presentation of this paper.}

\section{Preliminaries }\label{Section 2}
 
The aim of this section is to describe the notions we need and review some statements from previous work.

In this paper, let $K$ be an algebraically closed field endowed with a complete, 
non-archimedean absolute value $|~ |$. 
A variety over $K$ is an irreducible
separated reduced scheme of finite type over $K$ and a curve is a $1$-dimensional variety over $K$.

\subsection{Tropical curves and differential forms on Berkovich spaces}
Chambert--Loir and Ducros
introduced bigraded real-valued differential forms on Berkovich analytic spaces in their preprint \cite{CLD}. 
These are based on the definition of $(p, q)$-superforms on open subsets
of $\R^r$ by Lagerberg \cite{Lagerberg}. In this subsection, we give a short overview of these forms on the analytification of an algebraic curve.

Further, we review Poincar\'e duality and tropical modifications.
We restrict ourselves to the case of curves. 
The general case can be found, in a more conceptual way, in \cite[Section 2]{JSS} (for the tropical side) and \cite{JellThesis} (for the approach to 
forms on Berkovich spaces using $\A$-tropical charts).

\begin{defn} \label{defnd}
\begin{enumerate} [itemindent =*, leftmargin=0mm]
\item
For an open subset $U \subset \R^r$ denote by $\AS^{q}(U)$ the space of smooth real differential forms of degree $q$. 
The space of \emph{superforms of bidegree} $(p,q)$ on $U$ is defined as
\begin{align*}
\AS^{p,q}(U)
:= \Lambda^p {\R^r}^* \otimes_\R \AS^q(U).
\end{align*}
\item
There is a  \emph{differential operator} 
\begin{align*}
d'' \colon \AS^{p,q} (U) = \Lambda^p {\R^r}^* \otimes_\R \AS^q(U) \rightarrow \Lambda^p {\R^r}^* \otimes_\R \AS^{q+1}(U)= \AS^{p,q+1}(U)
\end{align*}
 which is given by $(-1)^p \id \otimes D$, where $D$ is the usual exterior derivative.

\item 
There is a wedge product  
\begin{align*}
\vedge: \AS^{p,q}(U) \times \AS^{p',q'}(U) &\rightarrow \AS^{p+p',q+q'}(U), \\
(\alpha \otimes \psi, \beta \otimes \nu) &\mapsto (-1)^{p' q} \alpha \vedge \beta \otimes \psi \vedge \nu ,
\end{align*}
which is, up to sign, induced by the usual wedge product.
\end{enumerate}
For all $p,q$ the functor $U\mapsto \AS^{p,q}(U)$ defines a sheaf on $\R^r$ and we have $\AS^{p,q}=0$ if $\max(p,q)>r$.
\end{defn}

\begin{defn} \label{defn:Tr} 
Let $\TT: = [-\infty , \infty)$ and equip it with the topology of a half open interval. 
Then $\TT^r$ is equipped with the product topology. 

A \emph{one dimensional polyhedral complex $\CS$ in $\R^r$}  
is a finite set of closed intervals, half lines, lines  in $\mathbb{R}^r$ (called \emph{edges}), 
and points (called \emph{vertices})
such that for each edge all endpoints are vertices and the intersection of two edges is empty or a vertex. 
For each $\sigma\in \CS$ we define 
$\Linear (\sigma):=\{\lambda\cdot(z-z')|\lambda\in\R,z,z'\in\sigma 
\} \subset \R^r$.
$\CS$ is called \emph{$\R$-rational} if $\Linear(\sigma) \cap \Z^r \neq \{0\}$ for all edges $\sigma \in \CS$.
The \emph{support} of such a complex $\CS$, which we denote by $\vert \CS\vert$, is the union of its edges and vertices. 
We only consider non-trivial one dimensional polyhedral complexes, i.e.~complexes which do not only consist of vertices. 
A \emph{polyhedral $\R$-rational curve $Y$ in $\T^r$} is the topological closure of $\vert \CS \vert$ 
for a one dimensional $\R$-rational polyhedral complex $\CS$. 
A \emph{polyhedral structure} $\CS$ on $Y$ is a one dimensional $\R$-rational polyhedral complex $\CS'$ 
such that $\vert \CS' \vert = Y \cap \R^r$ plus the vertices at infinity 
(i.e.~points in $Y \setminus \R^r$). 

Let $Y$ be a polyhedral $\R$-rational curve and $\CS$ a polyhedral structure on $Y$. 
$\CS$ is called \textit{weighted} if it is equipped with a \emph{weight} $m_\sigma  \in \Z$ on each edge $\sigma \in \CS$.
It is called \emph{balanced} if for all vertices $z$, which are not at infinity, we have 
\begin{align*}
\sum \limits_{\sigma: z \in \sigma} m_\sigma \nu_{z, \sigma} = 0,
\end{align*}
where $\nu_{z, \sigma}$ is the unique primitive vector in $\Z^r \cap \Linear (\sigma)$ pointing outwards of $\sigma$ at $z$.

A \emph{tropical curve} $Y$ is a $\R$-rational polyhedral curve with a balanced weighted polyhedral structure $(\CS, m)$,
up to weight preserving subdivision of $\CS$. 
Such a $(\CS, m)$ is then called a \emph{weighted polyhedral structure} on $Y$.
\end{defn}

\begin{defn}
Let $Y$ be $\R$-rational polyhedral curve in $\T^r$ and $\Omega$ an open subset of $Y$. 
If $\Omega \subset \R^r$, then a superform of bidegree $(p,q)$ on $\Omega$ is given by a superform $\alpha \in \AS^{p,q}(U)$ for $U$ an open
subset of $\R^r$ such that $\Omega = Y \cap U$. 
We can view $\alpha$ as a map $U \times (\R^r)^p \times (\R^r)^q \to \R$
and for $\sigma \in \CS$ we can consider its restriction to 
$(\Omega \cap \sigma) \times \Linear(\sigma)^p \times \Linear(\sigma)^q \subset U \times (\R^r)^p \times (\R^r)^q$. 
Two forms are identified if for some (and then all) polyhedral structures $\CS$ on $Y$ and all $\sigma \in \CS$ 
their restrictions to $(\Omega \cap \sigma) \times \Linear (\sigma)^{p} \times \Linear(\sigma)^{q}$ agree.

For general $\Omega$, a form $\alpha$ on $\Omega$ is given by a form $\alpha'$ on $\Omega' := \Omega \cap \R^r$,
satisfying the following \emph{boundary conditions}: 
For each $x \in \Omega \setminus \Omega'$, there exists a neighborhood $\Omega_x$ of $x$ in $\Omega$ such that 
\begin{enumerate}
\item
if $(p,q) = (0,0)$, then $\alpha' |_{\Omega_x \cap \R^r}$ is constant (note here that $\alpha'$ is indeed a function on $\Omega'$ in that case); 
\item
otherwise, we require $\alpha'|_{\Omega_x \cap \R^r} = 0$.
\end{enumerate}

We denote the space of such $(p,q)$-forms on $\Omega$ by $\AS^{p,q}(\Omega)$ and by $\AS^{p,q}_c(\Omega)$ the space of 
forms with compact support. 

The presheaf $\Omega \mapsto \AS^{p,q}(\Omega)$ is indeed a sheaf on $Y$. 

We define the differential and the wedge product by applying the respective operation to the forms on $U$ representing forms on $\Omega$. 
\end{defn}

\begin{bem}
The differential $d''$ and the wedge product are well defined on $\AS^{p,q}(\Omega)$ resp.~$\AS^{p,q}_c(\Omega)$. 

While the definition of forms on $Y$ may seem very ad hoc, they give a good notion of forms on tropical curves in $\T^r$. 
We refer to \cite{JSS} for more conceptual definitions which also work for higher dimensional tropical varieties. 
Theorem 3.22 in \emph{loc. cit.} can be seen as evidence that these definitions give a good definition of forms.
\end{bem}

\begin{bem}
Let $Y$ be a tropical curve and $\Omega$ an open subset of $Y$. 
There is a canonical non-trivial integration map $\int\colon \AS^{1,1}_c(\Omega) \rightarrow \R$
which satisfies $\int d'' \alpha = 0$ for all $\alpha \in \AS^{1,0}_c(\Omega)$ \cite[Definition 4.5 \& Theorem 4.9]{JSS}.
Note here that open subsets of tropical curves are tropical spaces in the sense of \cite[Definition 4.8]{JSS}. 
\end{bem}

\begin{defn}
Let $U' \subset \T^{r'}$ and $U \subset \T^r$ be open subsets.
An \emph{extended linear resp.~affine map} $F\colon U' \rightarrow U$ is the continuous extension (which may not always exist (!)) 
of a linear resp.~affine map $F\colon U'\cap \R^{r'} \to U\cap\R^r$. 
\end{defn}

\begin{bem}
There is a well defined, functorial pullback of superforms along extended affine maps $F$ which commutes with $d''$ and $\vedge$.
If $F$ maps a polyhedral curve $Y_1$ to a polyhedral curve $Y_2$, then this induces a pullback 
$F^* \colon \AS^{p,q}_{Y_2} \rightarrow F_*\AS^{p,q}_{Y_1}$. 
\end{bem}

\begin{satz}[Poincar\'e lemma]\label{Thm Poincare Lemma one vertex}
Let $Y$ be a tropical curve and $\Omega$ a connected open subset of $Y$, 
which, for some polyhedral structure $\CS$ on $Y$, contains at most one vertex.
Let $\alpha \in \AS^{p,q}(\Omega)$ such that $q > 0$ and $d'' \alpha = 0$. 
Then there exists $\beta \in \AS^{p,q-1}(\Omega)$ such that $d'' \beta = \alpha$. 
\end{satz}
\begin{proof}
If $\Omega$ contains a vertex $z$ which is not at infinity, then it is polyhedrally 
star shaped with center $z$ in the sense of \cite[Definition 2.2.11]{Jell}. 
If $\Omega$ contains no vertex, then it is just an open line segment and polyhedrally 
star shaped with respect to any of its points.
Thus if $\Omega$ contains no vertex at infinity, the result follows from \cite[Theorem 2.16]{Jell}. 
If the vertex is at infinity, $\Omega$ is a half open line with vertex 
at infinity, thus basic open in the sense of \cite[Definition 3.7]{JSS} and the statement follows from \cite[Theorem 3.22 \& Proposition 3.11]{JSS}. 
\end{proof}

\begin{defn} \label{Def. Trop. Grundlagen}

We denote by $\A^{r, \an}$ the \emph{Berkovich analytification of} $\A^r_K$ which is the set of multiplicative seminorms on $K[T_1,\dots,T_r]$,
which extend the absolute value on $K$, 
endowed with the weakest topology such that for every $f\in K[T_1,\dots,T_r]$ the map $\A^{r, \an}\to \R,~\rho\mapsto \rho(f)$ is continuous. 
A basis of open subsets of $\A^{r, \an}$ is given by sets of the form
$
\{ \rho\in\A^{r, \an}\;\mid  a_i < \rho(f_i )< b_i \}
$
for $a_i, b_i \in \R \cup \{\infty\}$ and $f_i \in K[T_1,\dots,T_r]$. 
We call sets of this form \emph{standard open subsets} of $\A^{r,\an}$.

In general, we denote for an algebraic variety $X$ by $\Xan$ its analytification in the sense of Berkovich \cite{BerkovichSpectral}.
\end{defn}

\begin{defn}\label{Def. transition map}
Let $Z$ be a closed subvariety of $\A^r = \Spec (K[T_1,\dots,T_r])$. 
Then we define $\Trop(Z)$ to be the image of $Z^{\an}$ under the tropicalization map
\begin{align*}
\trop\colon \A^{r, \an} &\rightarrow \T^r, \\
\rho &\mapsto (\log (\rho(T_1)), \dots, \log (\rho(T_r))).
\end{align*}
Let $X$ be a curve, $U$ an open subset of $X$ and $\varphi\colon U \rightarrow \A^r$ a closed embedding 
such that $\varphi(U)\cap \G^r_m\neq \emptyset$. 
We say that $\varphi$ \textit{is given by} $f_1,\ldots,f_r\in \OS_X(U)$ if the corresponding $K$-algebra homomorphism 
$\varphi^\sharp\colon K[T_1,\ldots,T_r]\to \OS_X(U)$ maps $T_i$ to $f_i$. 
Then $\varphi_{\trop} :=  \trop\circ\varphi^{\an}\colon U^{\an} \rightarrow \T^r$ is given by $\rho \mapsto (\log (\rho(f_1)), \dots, \log (\rho(f_r)))$. 
The map $\varphi_{\trop}$ is proper in the sense of topological spaces.
	
Further, we write $\Trop_\varphi(U):=\varphi_{\trop}(U^{\an})$. 
Note that $\Trop_\varphi(U)\cap \R^{r}$ is equal to the support of a one dimensional $\R$-rational polyhedral complex 
\cite[Theorem 10.14]{Gubler2} and 
$\Trop_\varphi(U)$ is its closure \cite[Lemma 3.1.4]{JellThesis}. Hence, $\Trop_\varphi(U)$ is an $\R$-rational polyhedral curve.
There is also a canonical way to associate positive weights to $\Trop_\varphi(U)$.
How to do this for $\Trop_\varphi(U) \cap \R^r$ is explained in \cite[13.10]{Gubler2}.
Since all our edges are contained in $\R^r$, we simply take the weights defined there.

Another embedding $\varphi'\colon U' \rightarrow \A^{r'}$ is called a \emph{refinement} of $\varphi$ if $U' \subset U$ and 
there exists a torus equivariant map $\psi\colon \A^{r'} \rightarrow \A^r$ such that $\varphi|_{U'} = \psi \circ \varphi'$. 
If that is the case, we obtain an extended linear map $\Trop(\psi)\colon \Trop_{\varphi'}(U') \rightarrow \Trop_{\varphi}(U)$ 
which satisfies $\varphi_{\trop}=\Trop(\psi)\circ \varphi'_{\trop}$ on $(U')^{\an}$. 
Since we have $\Trop_{\varphi'}(U') =\varphi'_{\trop}((U')^{\an})$, the map $\Trop(\psi)$ is independent of $\psi$. 
We call this map the \emph{transition map} between $\varphi$ and $\varphi'$. 
\end{defn}

Note that if for $i = 1,\dots,n$ we have closed embeddings $\varphi_i \colon U_i \to \A^{r_i}$, then
$\varphi_1 \times \dots \times \varphi_n \colon \bigcap_{i = 1}^n U_i \to \prod_{i=1}^n \A^{r_i}$ is a 
common refinement of all the $\varphi_i$.

\begin{satz}\label{Theorem tropical curve}
The space $\Trop_\varphi(U)$ together with the weights mentioned above is 
a tropical curve. 
\end{satz}
\begin{proof}
For $\Trop_\varphi(U) \cap \R^r$, this theorem is due to many people, including Bieri, Groves, Sturmfels and Speyer. 
For a proof see \cite[Theorem 13.11]{Gubler2}.
Since our weights are the same as in \cite[Theorem 13.11]{Gubler2} and the balancing condition is only required at points in $\R^r$, 
the theorem follows.
\end{proof}

\begin{defn} \label{defn:GATcharts}
Let $X$ be an algebraic curve.
An \emph{$\A$-tropical chart} is given by a pair $(V, \varphi)$, 
 where $V \subset \Xan$ is an open subset and 
$\varphi\colon U \to\A^r$ is a closed embedding of an affine open subset $U$ of $X$
 such that $V = \varphi_{\trop}^{-1}(\Omega)$ for an open subset $\Omega \subset \Trop_\varphi(U)$.
 
Another $\A$-tropical chart $(V', \varphi')$ is called an \emph{$\A$-tropical subchart} of $(V, \varphi)$
 if $\varphi'$ is a refinement of $\varphi$ and $V' \subset V$. 
\end{defn}
Note that if $(V, \varphi)$ is a tropical chart and $\varphi' \colon U' \to \A^{r'}$ is a refinement of $\varphi$ 
such that $V \subset U'^{\an}$, 
then $(V, \varphi')$ is an $\A$-tropical chart, 
thus in particular a subchart of $(V, \varphi)$. 
To see this, let $\Trop(\psi)$ be the transition map between $\varphi$ and $\varphi'$ and $\Omega$ such that 
$V = \varphi_{\trop}^{-1}(\Omega)$. 
Then $V = \varphi'^{-1}_{\trop}(\Trop(\psi)^{-1}(\Omega))$.

$\A$-tropical charts form a basis of the topology of $\Xan$ \cite[Lemma 3.2.35 \& Lemma 3.3.2]{JellThesis}.
We will show this for $X = \A^1$ in Lemma \ref{lemma existence chart}.

\begin{defn}\label{Def. forms} 
Let $X$ be an algebraic curve and $V$ an open subset of $\Xan$. 
An element of $\AS_X^{p,q}(V)$ is given by a family $(V_i, \varphi_i, \alpha_i)_{i \in I}$ such that
\begin{enumerate}[itemindent=*, leftmargin=0mm]
\item for all $i \in I$ the pair $(V_i, \varphi_i)$ is an $\A$-tropical chart and $\bigcup _{i \in I} V_i = V$;
\item for all $i \in I$ we have $\alpha_i \in \AS^{p,q}_{\Trop_{\varphi_i}(U_i)}( \varphi_{i,\trop}(V_i))$;
\item for all $i, j \in I$ 
there is a cover of $V_i \cap V_j$ by common tropical subcharts $(V_{ijl}, \varphi_{ijl})_{l \in L}$ 
such that for all $l$ 
the forms $\alpha_i$ and $\alpha_j$ agree when pulled back to 
$(V_{ijl}, \varphi_{ijl})$ via the corresponding transition maps.
\end{enumerate}
Another such family $(V_j, \varphi_j, \alpha_j)_{j \in J}$ defines the same form on $V$ if $(V_i, \varphi_i, \alpha_i)_{i \in I \cup J}$ still defines a form on $V$.

Let $V'\subset V$ and $\alpha\in\AS_X^{p,q}(V)$ given by $(V_i, \varphi_i, \alpha_i)_{i \in I}$.
The subset $V'$ can be covered by $\A$-tropical subcharts $(W_{ij},\varphi_{ij})$ of $(V_i,\varphi_i)$ \cite[Lemma 3.2.35]{JellThesis}.
We define $\alpha|_{V'}$ to be given by the family $(W_{ij},\varphi_{ij},\Trop(\psi_{i,ij})^*\alpha_i)_{ij}$. 
Note that $\alpha|_{V'}$ is independent of all choices.

Then  $V\mapsto\AS_X^{p,q}(V)$ defines a sheaf on $\Xan$, which we denote by  $\AS_X^{p,q}$. 
If the space of definition is clear, we often just use the notation $\AS^{p,q}$. 
By Theorem \ref{Theorem tropical curve}, we have $\AS_X^{p,q}=0$ if $\max(p,q)>1$.

The differential $d''$ and the wedge product carry over. 

For a tropical chart $(V, \varphi)$, we call the map $\AS^{p,q}(\varphi_{\trop}(V)) \rightarrow \AS^{p,q}(V)$ the 
\emph{pullback along} $\varphi_{\trop}$.  
\end{defn}

\begin{prop} \label{pullback injective}
Let $(V, \varphi)$ be an $\A$-tropical chart. 
Then the pullback along $\varphi$ is injective. 
Further, if $\alpha \in \AS^{p,q}(V)$ is given by $(V, \varphi, \alpha')$, then
$\varphi_{\trop}(\supp(\alpha)) = \supp(\alpha')$. 
\end{prop}
\begin{proof}
The first statement is shown in \cite[Lemma 3.2.42]{JellThesis}.
The second statement then follows from the first as in \cite[Corollaire 3.2.3]{CLD}.
\end{proof}

\begin{defn}\label{Defn nach forms}
For $V$ an open subset of $\Xan$, there is a non-trivial integration map $\int\colon \AS^{1,1}_c(V) \rightarrow \R$ which is compatible with pullback 
and satisfies $\int d'' \alpha = 0$ for all $\alpha \in \AS^{1,0}_c(V)$. For an explicit description we refer to \cite[Definition 3.2.58]{JellThesis}. 

Let $Y$ be either the analytification of an algebraic curve or a polyhedral curve and $V$ an open subset of $Y$. 
Then we denote the presheaf $V \mapsto \AS^{p,q}_{Y, c}(V)^* := \Hom_\R(\AS^{p,q}_{Y,c}(V), \R)$ by ${\AS^{p,q}_{Y, c}}^*$. 
Note that we do not put any topology on $\AS^{p,q}_Y$, thus the dual is always meant in the sense of linear algebra.
We write $\HH^{p,q}(V) := \HH^q(\AS^{p, \bullet}(V), d'')$ and $\HH^{p,q}_c(V) := \HH^q( \AS^{p, \bullet}_c(V), d'')$
and denote by $h^{p,q}(V)$ resp.~$h^{p,q}_c(V)$ the respective $\R$-dimensions. 
We denote by $d''^*$ the dual of the differential and use the notation 
$\LS_Y^p:=\ker(d''\colon\AS_Y^{p,0} \rightarrow \AS_Y^{p,1})$ and 
$\GS_Y^p := \ker ( d''^* \colon {\AS^{p,1}_{Y, c}}^* \rightarrow {\AS^{p,0}_{Y, c}}^*)$. 
\end{defn}
\begin{defn}
We denote by $h^{q}_{\sing}(V)$ resp.~$h^{q}_{c,\sing}(V)$ the $\R$-dimension of the singular cohomology 
$\HH^q_{\sing}(V,\R)$ resp.~$\HH^q_{c,\sing}(V,\R)$. 
\end{defn}

\begin{satz} \label{Identification with singular cohomology}
Let $X$ be an algebraic resp.~tropical curve and $V$ an open subset of $\Xan$ resp.~$X$.
Then $\LS^p_V \rightarrow (\AS^{p, \bullet}_V, d'')$ is an acyclic resolution of $\LS^p$. 
In particular, we have $\HH^q(V, \LS^p) = \HH^q((\AS^{p,\bullet}(V), d''))$ and $\HH^q_c(V, \LS^p) = \HH^q((\AS^{p,\bullet}_c(V), d''))$. 
Here $\HH^q_c$ denotes sheaf cohomology with compact support. 
We also have $\LS_V^0 = \underline{\R}$, where $\underline{\R}$ denotes the constant sheaf with stalks $\R$.
As a consequence, we obtain $h^{0,q}(V) = h^{q}_{\sing}(V)$ and $h^{0,q}_c(V) = h^q_{c,\sing}(V)$.
\end{satz}
\begin{proof}
The sheaves $\AS^{p,q}_X$ are fine for a tropical curve $X$ \cite[Lemma 2.15]{JSS}. 
As a consequence, the sheaves $\AS^{p,q}_X$ are also fine for an algebraic curve $X$ 
\cite[Lemma 3.2.17, Proposition 3.2.46 \& Lemma 3.3.6]{JellThesis}.
That $(\AS^{p, \bullet}_V, d'')$ is exact in positive degree follows from Theorem \ref{Thm Poincare Lemma one vertex} for the tropical case
and from \cite[Theorem 3.4.3]{JellThesis} for the algebraic case. 
That $\LS^0_V = \underline{\R}$ is shown in \cite[Lemma 3.4.5]{JellThesis}. 
The rest now follows from standard sheaf theory since $V$ is Hausdorff and locally compact.
The fact that $\AS^{p,q}_V$ are acyclic for both the functor of global section respectively global sections with compact support follows from
\cite[Chapter II, Proposition 3.5 \& Theorem 3.11]{Wells} and \cite[III, Theorem 2.7]{Iversen} and identification with singular cohomology comes from
\cite[Chapter III, Theorem 1.1]{Bredon}. 
\end{proof}

\begin{lem} \label{Lem compact support sheaf}
Let $Y$ be either the analytification of an algebraic curve or a polyhedral curve. 
Then the presheaf ${\AS^{p,q}_{Y,c}}^*$ is a sheaf and flasque. 
Further,
\begin{align}\label{Gl. Cohom.}
\HH^q( {\AS^{1-p, 1-\bullet}_{Y,c}}^*(V), d''^*) = 
\HH^{1-p,1-q}_c(V)^*
\end{align}
for every open subset $V$ of $Y$.
\end{lem}
\begin{proof}
It is a sheaf because  $\AS^{p,q}_{Y}$ admits partitions of unity and it is flasque 
because for all $W \subset V$ the map $\AS^{p,q}_{Y, c}(W) \rightarrow \AS^{p,q}_{Y,c}(V)$ is injective.
The second assertion is true since dualizing is an exact functor. 
\end{proof}

\begin{defn}\label{Def. PD}
Let $X$ be an algebraic or tropical curve and $V$ an open subset (of $\Xan$ in the algebraic case). 
We define 
\begin{align*}
\PD \colon \AS^{p,q}(V) \to \AS^{1-p,1-q}_c(V)^*,~
\alpha \mapsto ( \beta \mapsto \varepsilon \int \alpha \vedge \beta)
\end{align*}
where $\varepsilon = 1$ if $p = q = 0$ and $\varepsilon=-1$ else. 
The morphism $\PD$ defined above induces a morphism of complexes 
$\PD \colon \AS^{p, \bullet}(V)\to \AS^{1- p, 1-\bullet}_c(V)^*$
where the complex $\AS^{p, \bullet}(V)$ is equipped  with $d''$ and $\AS^{1-p, 1-\bullet}_c(V)^*$ with its dual map $(d'')^*$. 

Hence, we get a morphism $\PD \colon \HH^{p,q}(V) \rightarrow \HH^{1-p,1-q}_c(V)^*$ by (\ref{Gl. Cohom.}). 
We say that $V$ \emph{has $\PD$} if the map on cohomology is an isomorphism for all $(p,q)$.
\end{defn}

\begin{lem} \label{lemma PD injective}
The map $\PD$ defined above induces a monomorphism $\LS^p \rightarrow \GS^{1-p}$.
\end{lem}
\begin{proof}
That $\PD$ maps $\LS^p$ to $\GS^{1-p}$ follows because $\PD$ is a morphism of complexes. 
To show that this map is injective it is sufficient to show that $\PD \colon \AS^{p,0} \rightarrow {\AS^{1-p,1}_c}^*$ is injective,
i.e.~that for all $\alpha \in \AS^{p,0}(V) \setminus \{0\}$ there exists $\beta \in \AS^{1-p,1}_c(V)$ such that $\int_V \alpha \vedge \beta \neq 0$. 
In the tropical situation, there exists an open subset $\Omega$ which is contained in an edge $\sigma$ and 
a coordinate $x$ on $\sigma$ such that 
$\alpha \vert_\Omega = \pm f$ resp.~$\alpha \vert_\Omega = f  d'x$ for $f > 0$. 
Then letting $g$ be a bump function with compact support in $\Omega$ and $\beta = \pm g d'x \otimes d''x$ resp.~$\beta = g d''x $ suffices

In the algebraic situation, let $\alpha$ be given by $(V_i, \varphi_i, \alpha_i)$. 
Choose $i$ such that $\alpha_i \neq 0$ and $\beta_i \in \AS^{1-p,1}_c(\varphi_{i, \trop}(V_i))$ 
with $\int_{\varphi_{i, \trop}(V_i)} \alpha_i \vedge \beta_i \neq 0$. 
By definition of integration, we have $\int_V \alpha \vedge \varphi_{i, \trop}^*(\beta_i) = \int_{\varphi_{i, \trop}(V_i)} \alpha_i \vedge \beta_i$ 
which proves the claim.
\end{proof}

\begin{defn} \label{Def. val, dim, glatt}
Let $z$ be a vertex of a tropical curve $Y$ and let $\sigma_0,\dots,\sigma_k$ be the edges which contain $z$.
We define $\val(z) := k+1$. 
Further, we define $\dim(z) := \dim (\Linear(\sigma_0) + \ldots + \Linear(\sigma_k))$ if $z$ is not at infinity 
and $\dim(z) = 0$ if $z$ is at infinity.

A tropical curve $X$ is called \emph{smooth} if all weights are $1$ and 
for every vertex $z$ we have $\val(z) = \dim(z) +1 $.
\end{defn}

\begin{satz} \label{PD open subsets of smooth tropical curves}
Open subsets of smooth tropical curves have $\PD$. 
\end{satz}
\begin{proof}
Let $Y$ be a smooth tropical curve and $\Omega$ an open subset. 
Let $\CS$ be a polyhedral structure on $Y$. 
Since tropical manifolds in the sense of \cite[Definition 4.15]{JSS} have $\PD$ by \cite[Theorem 4.33]{JSS},
it is enough to show that any point $z \in \Omega$ has a neighborhood which 
is either isomorphic to an open subset of $\T$ or to an open subset of a Bergman fan $B(M)$ of a matroid $M$. 
If $z$ is not a vertex, it has a neighborhood which is isomorphic to an open interval, 
thus to an open subset of $\T$. 
If $z$ is a vertex at infinity, it has a neighborhood isomorphic to $[-\infty, b) \subset \T$ for some $b \in \R$. 
Now let $z$ be a vertex which is not at infinity and let $k$ be as in Definition \ref{Def. val, dim, glatt}. 
Let $U_{2,k}$ be the uniform matroid of rank 2 on $k$ Elements, i.e.~the base set is $\{1,\dots,k\}$ 
and the rank function is $A \mapsto \max \{ \# A, 2\}$. 
Following the construction of the Bergman fan in \cite[\S 2.4]{Shaw:IntMat} we find that $B(U_{2,k})$ 
is the fan whose rays are spanned by $-e_1,\dots,-e_k$ and  $\sum_{i=1}^ke_i$ where $e_i$ denotes 
the i-th unit vector in $\R^k$. 
Now, after translation of $z$ to the origin, $\nu_{z, \sigma_i} \mapsto  e_i$ for $i = 1,\dots,k$ 
provides a linear isomorphism of a neighborhood of $z$ with an open neighborhood of the origin in $B(U_{2,k})$.
Note here that $- \sigma_{z, \sigma_0} \mapsto \sum e_i$ by the balancing condition and 
that $\nu_{z, \sigma_i}$  for $i=1,\dots,k$ are linearly independent since $\dim(z) = \val(z) + 1$.  
\end{proof}

\begin{Const}  \label{Tropicalmodification}
We now describe the operation of tropical modification (for a more detailed introduction see \cite{BIMS15}).
Let $X \subset \T^r$ be a tropical curve and $P\colon X \rightarrow \R$ a continuous, piecewise affine function with integer slopes. 
The graph $\Gamma_X(P)$ of $P$ is a polyhedral $\R$-rational curve in $\T^{r+1}$.
Choosing a polyhedral structure $\CS$ on $X$ such that $P$ is affine on every edge and defining
the weight $m_{\Gamma_{\sigma}(P)} := m_\sigma$ for every $\sigma\in \CS$ makes $\Gamma_X(P)$ into a weighted $\R$-rational polyhedral curve.
It is however not balanced, because $P$ is only piecewise affine. 
Let $z \in X$ be a point where $P$ is not affine. 
If we add a line $\sigma_z := [ (z, P(z)), (z, -\infty) ]$, then there is a unique
weight $m_{\sigma_z}$ to make $\Gamma(P) \cup \sigma_z$ balanced at $z$. 
If we do this for every such $z\in X$, we obtain a tropical curve $Y$. 
The projection $\pi \colon \T^{r+1} \rightarrow \T^{r}$ restricts to a map $\delta \colon Y \rightarrow X$ and 
we call $\delta$ a \emph{tropical modification}.

Note that $\delta$ is a proper map in the sense of topological spaces.
\end{Const}

\begin{Prop} \label{prop:closedmodification}
Let $\delta\colon Y \rightarrow X$ be a tropical modification.
Let $V$ be an open subset of $X$ and $W = \delta^{-1}(V)$.  
Then  
\begin{align*}
\delta^*\colon \HH^{p,q}(V) \rightarrow \HH^{p,q}(W) \qquad  \text{ and} \qquad  
\delta^*\colon \HH^{p,q}_c(V) \rightarrow \HH^{p,q}_c(W)
\end{align*}
are isomorphisms which are compatible with the Poincar\'e duality map.
\end{Prop}
\begin{proof}
The first statement follows from \cite[Theorem 4.13]{Shaw:Surfaces}
using the identification of $\HH^{p,q}$ with tropical cohomology  \cite[Theorem 3.22 \& Proposition 3.24]{JSS}.
For a proof which is closer to our setting and also takes $\HH^{p,q}_c$ into account, we refer to \cite[Corollary 1.58]{Smacka}. 

That this is compatible with the $\PD$ map is just saying that integration of a $(1,1)$-form with compact support commutes with pullback along $\delta$. 
This follows from the tropical projection formula \cite[Proposition 3.10]{Gubler}.
\end{proof}

\subsection{Non-archimedean Mumford curves}\label{Section 2.2}
In this subsection, we recall the definition of Mumford curves and give a further characterization of them
which is needed in Sections \ref{Section 4} and \ref{Section 5}.

\begin{defn}\label{Def. Genus}
We say that an analytic space $Y$ is \emph{locally isomorphic to $\mathbb{P}^{1, \an}$} if there 
is a cover of $Y$ by open subsets which are isomorphic to open subsets of $\mathbb{P}^{1, \an}$ in the sense of analytic spaces.

Let $X$ be a smooth algebraic curve over $K$ and $x\in X^{\an}$.
We denote by $\mathscr{H}(x)$ the completed residue field at $x$ and by  $\widetilde{\mathscr{H}}(x)$ its residue field.
The point $x$ is said to be \emph{of type 2} if $\widetilde{\mathscr{H}}(x)$ is of transcendence degree $1$ over $\tilde{K}$, 
the residue field of $K$. 
If this is the case, the \emph{genus} of $x$ is defined as the genus of the smooth projective $\tilde{K}$-curve with function field 
$\widetilde{\mathscr{H}}(x)$.

Note that there are only finitely many points of type $2$ of positive genus in $\Xan$ \cite[Remark 4.18]{BPR2}.
\end{defn}

\begin{prop} \label{Prop Thm Mumford curve}
Let $X$ be a smooth curve and $V \subset \Xan$ an open subset. 
Then $V$ is locally isomorphic to $\mathbb{P}^{1, \an}$ if and only if it does not contain any point of type $2$ with positive genus.
\end{prop} 

\begin{proof}
By \cite[Proposition 2.3]{BR}, every type $2$ point of $\mathbb{P}^{1, \an}$ has genus zero. 
Hence, if $V$ is locally isomorphic to $\mathbb{P}^{1, \an}$, all its type $2$ points have genus $0$.

Now, we assume that every type $2$ point in $V$ has genus zero. 
Then \cite[Proposition 2.2.1]{BerkoIntegration} says that we have an open covering of $V$ by open discs, 
open annuli and sets which are isomorphic to $\mathbb{P}^{1, \an}$ without the disjoint union of a finite number of closed discs. 
Note that we have used that the definition of the genus of a point in \cite{BerkoIntegration} is equal to the one in 
Definition \ref{Def. Genus}  \cite[p. 31]{BerkoIntegration}.
Hence, $V$ is locally isomorphic to $\mathbb{P}^{1, \an}$.
\end{proof}

\begin{defn}\label{Definitions Mumford cruve}
A smooth projective curve  $X$ of genus $g\geq 1$ is called a \textit{Mumford curve} 
if there is a semistable model $\XS$ such that all irreducible components of the special fibre are rational (cf.~\cite[Theorem 4.4.1]{BerkovichSpectral}).
\end{defn}

Note that there exist no Mumford curves over a trivially valued field $K$ since otherwise $K^\circ = K$ and 
so the only semistable model of $X$ is $X$ itself, which cannot be rational due to $g\geq 1$.

\begin{satz}\label{Satz Mumford curve}
Let $X$ be a smooth projective curve over $K$ of genus $g$. Then the following properties are equivalent:
	\begin{enumerate}
		\item $X$ is a Mumford curve or is isomorphic to $\mathbb{P}^{1}$.
		\item $\Xan$ is locally isomorphic to $\mathbb{P}^{1, \an}$.
		\item $h^{0,1}(\Xan) = g$.
	\end{enumerate} 
\end{satz}
\begin{proof}
If $g=0$, then $X$ is isomorphic to $\mathbb{P}^1$
 and all three properties are satisfied. Indeed, the third one is true since 
$h^{0,1}(\mathbb{P}^{1, \an})=h^1_{\sing}(\mathbb{P}^{1, \an})$ by 
Theorem \ref{Identification with singular cohomology} and $\mathbb{P}^{1, \an}$ is contractible.

If $g\geq 1$, property i) implies ii) by \cite[Theorem 4.4.1]{BerkovichSpectral}.  
Note here that 
for any analytic space $Y$, the topological universal cover $\pi \colon Z \to Y$ of $Y$ is given by the analytic structure 
which makes $\pi$ into a local isomorphism. 
Thus $\Xan$ is locally isomorphic to its universal cover.

On the other hand, if ii) is satisfied, we know from Proposition \ref{Prop Thm Mumford curve} that every type 2 point in $\Xan$ has genus zero. 
Let $\XS$ be any semistable model.
Then every irreducible component of the special fibre corresponds to a type 2 point $x\in \Xan$, and we denote this component by $C_x$. 
The curve $C_x$ is birationally equivalent to the smooth projective $\tilde{K}$-curve with function field 
$\widetilde{\mathscr{H}}(x)$ by \cite[Proposition 2.4.4]{BerkovichSpectral}. 
We know that the latter curve is of genus zero, and so $C_x$ is as well. Thus, every irreducible component $C_x$ is rational.

Since the skeleton of a Berkovich space is a deformation retract, 
iii) is equivalent to the skeleton of $\Xan$ having first Betti number equal to $g$  by Theorem \ref{Identification with singular cohomology}.
Thus we know from \cite[Theorem 4.6.1]{BerkovichSpectral}, 
that we have  $h^{0,1}(\Xan) = g$ if and only if $X$ is a Mumford curve or a principal homogeneous space over a Tate elliptic curve. 
The first sentence in the proof of \cite[Lemma 4.6.2]{BerkovichSpectral} shows that if $K$ is algebraically closed, 
the only principal homogeneous space over any Tate elliptic curve is the curve itself. 
Since Tate elliptic curves are indeed Mumford curves, we have equivalence of i) and iii) also for $g \geq 1$.
Thus all properties are equivalent.
\end{proof}

\section{Cohomology of open subsets of the Berkovich affine line} \label{section modifications}

The goal of this section is to get a better description of the cohomology of a basis of open subsets of $\A^{1,\an}$ 
(cf.~Theorem \ref{thm Cohom. lineare trop.}). 
We use this description to prove Poincar\'e duality for a special class of open subsets of the analytification 
$\Xan$ of a smooth algebraic curve $X$ in Section \ref{Section 4}.

\subsection{Tropicalization of linear embeddings and refinements} 
In this section, we consider a special class of embeddings of the affine line into affine spaces, which are called linear embeddings. 
We show that their tropicalizations are smooth tropical curves and 
that their refinements induce tropical modifications (Theorem \ref{theorem tropical manifold} and Theorem \ref{theorem refinement modification}).

\begin{defn}
A closed embedding $\varphi\colon \A^1 \rightarrow \A^r$ is called a \emph{linear embedding} 
if $\varphi$ is given by linear polynomials $(x-a_i)_{i\in [r]}$, where $[r]:=\{1,\ldots,r\}$. 

Another linear embedding $\varphi'\colon \A^1 \rightarrow \A^{r'}$ is called a \emph{linear refinement} 
if $\varphi =\pi \circ \varphi'$, where $\pi$ is the projection to a set of coordinates. 
\end{defn}
\begin{lem}\label{Prop weights}
Let $\varphi\colon \A^1 \rightarrow \A^r$ be a linear embedding. 
Then all weights on $\Trop_\varphi(\A^1)$ are equal to 1. 
\end{lem}
\begin{proof}
Let $\varphi$ be given by $x - a_1,\ldots,x-a_r$ and  $\sigma$ an edge of $\Trop_\varphi(\A^1)$. 
Then there exists some coordinate $i$ such that the restriction of the $i$-th coordinate function to $\sigma$ is not constant. 
Obviously, $\varphi$ is a refinement of the linear $\A$-tropical chart $\varphi'\colon \A^1 \rightarrow \A^1$ 
which is given by $x - a_i$. 
The weights on $\Trop_{\varphi'}(\A^1)$ are obviously all $1$ and thus so is the one on $\sigma$ by the Sturmfels-Tevelev multiplicity formula 
and the construction of the pushforward of tropical cycles \cite[4.10 \& Proposition 4.11]{Gubler}.
Note here, that by the choice of $i$, $\sigma$ is not contracted to a point when projecting from $\Trop_{\varphi}(\A^1)$ to $\Trop_{\varphi'}(\A^1)$.
\end{proof}

\begin{lem} \label{lem seminorm}
Let $R$ be a commutative ring with $1$ and $\vert \, . \, \vert$ be a non-archimedean multiplicative seminorm on $R$. 
Let $x, a_i, a_j, b \in R$ such that  $\vert x- a_i\vert \leq \vert x - a_j \vert$ and $\vert x - a_j \vert \neq \vert b - a_j \vert$. 
Then we have 
\begin{align*}
\max ( \vert x - a_i \vert, \vert b - a_i \vert) = \max (\vert x - a_j \vert, \vert b - a_j \vert). 
\end{align*}
\end{lem}
\begin{proof}
If the maximum on the left hand side is attained uniquely, 
then both sides equal $\vert x - b \vert$ by the ultrametric triangle inequality, and so the lemma holds.

If not, we have 
\begin{align*}
\vert x-b\vert & \leq \max (\vert x-a_i\vert, \vert b-a_i\vert)=\vert x-a_i\vert \\ & \leq \vert x-a_j\vert \leq \max (\vert x-a_j\vert, \vert b-a_j\vert)=\vert x-b\vert,
\end{align*}
and so equality holds as well.
\end{proof}

We fix a closed embedding $\varphi\colon \A^1 \to \A^r$ which is given by linear polynomials $x- a_1, \dots, x - a_r$. 
For $b \in K$ we want to understand the behavior of the refinement $\varphi'\colon \A^1 \rightarrow \A^{r+1}$ which is given by 
$x- a_1, \dots, x - a_r, x- b$. 
We will for the moment assume that $b \neq a_j$  for all $j\in [r]$. 

\begin{defn}\label{definition P}
We define a map $P := P_{a_1,\dots,a_r, b} \colon \Trop_{\varphi}(\A^1) \rightarrow \R$
in the following way: For $z \in \Trop_{\varphi}(\A^1)$ choose $i \in [r]$ such that $z_i \leq z_j$ for all $j \in [r]$ 
and define $P(z) :=  \max(z_i, \log \vert b - a_i \vert )$.
\end{defn}

The next proposition shows the basic properties of $P$.

\begin{prop} \label{Lemma iteriertes Maximum} \label{theorem iterated maximum}
In the situation above, we have: 
\begin{enumerate}
\item
$P(z)$ is well defined, independent of the choice of $i$. 
\item
For $j \in [r]$ we have 
\begin{align*}
P(z) = \max (z_j, \log \vert b - a_j \vert ) 
\end{align*}
if the maximum on the right hand side is attained uniquely.
\end{enumerate}
\end{prop}
\begin{proof}
The second statement follows from Lemma \ref{lem seminorm} by choosing $R = K[x]$ and 
letting $\vert \, . \, \vert$ be a point in $\A^{1,\an}$ which maps to $z$. 
The first then follows from the second and the fact that for two minima $z_i, z_j$
the equality $\max (z_i, \log \vert b - a_i \vert)=\max (z_j, \log \vert b - a_j \vert)$ is trivial if 
$z_i = \log \vert b - a_i \vert = z_j = \log \vert b - a_j \vert$.   
\end{proof}

\begin{lem}\label{Lemma P properties}
The function $P$ is continuous.
It is affine at every point except $(\log \vert a_i - b\vert)_{i \in [r]}$. 
In particular, $P$ is piecewise affine.
\end{lem}
\begin{proof} 
Let $\CS$ be a polyhedral structure on $\Trop_\varphi(\A^1)$ such that for each $\sigma \in \CS$ there 
exists $i_\sigma \in [r]$ such that for all $z \in \sigma$ we have $z_{i_\sigma} \leq z_j$ for all $j \in [r]$ . 
Then $P|_{\sigma}$ is continuous by definition, thus $P$ is continuous. 

If $z \neq (\log \vert a_i - b\vert)_{i \in [r]}$, then there exists $j\in [r]$ such that $\max (z_j, \log \vert b - a_j \vert )$ is attained uniquely. 
Since this maximum is then attained uniquely for all $z'$ in a neighborhood of $z$, 
Proposition \ref{Lemma iteriertes Maximum} ii) shows that $P$ is either constant or the projection to the $j$-th coordinate 
on this neighborhood, so in particular affine. 
\end{proof}

\begin{satz} \label{theorem refinement modification}
Let $\varphi\colon \A^1 \to \A^r$ be a linear embedding  
and $\varphi'\colon \A^1 \to \A^{r+1}$ a linear refinement. 
We consider the commutative diagram 
\begin{align*}
\begin{xy}
\xymatrix{
	&    \A^{r+1}  \ar[d]^{\pi}  \\
	\A^1 \ar[ru]^{\varphi'} \ar[r]^{\varphi} & \A^{r}   
}
\end{xy} 
\end{align*} and the map 
$\Trop(\pi)\colon\Trop_{\varphi'}(\A^1)\to \Trop_{\varphi}(\A^1)$
induced by the projection $\TT^{r+1}\to \TT^r$. Then $\Trop(\pi)$ is a tropical modification. 
\end{satz}

\begin{proof} 
Let $\varphi$ be given by $(x- a_1), \dots ,(x -  a_r)$ and $\varphi'$ additionally by $(x - b)$.
If there exists $j \in [r]$ such that $b = a_j$, then $z \mapsto (z, z_j)$ is an extended linear map which is an inverse of $\Trop(\pi)$. 
In particular, $\Trop(\pi)$ is an isomorphism and we are done. 
Thus, we may assume that $b \neq a_j$ for all $j \in [r]$.
 
In the following, we write $X$ (resp.~$X'$) for $\Trop_{\varphi}(\A^1)$ (resp.~$\Trop_{\varphi'}(\A^1))$, 
use $P$ as defined in Definition \ref{definition P} and 
  use the notation $z_{s}$ for the  point $(\log\vert b-a_i\vert)_{i\in [r]} \in \R^r $. 
We want to show that $X'$ is the completion of the graph 
$\Gamma_X(P) \subset X \times \R \subset \T^{r+1}$ to a tropical curve as explained in Construction \ref{Tropicalmodification}. 

We show that if $z \in X \setminus \{ z_s\}$, the unique preimage of $z$ under $\Trop(\pi)\colon X' \rightarrow X$
is the point $(z, P(z))$ and that the preimage of $z_s$ is the line $[(z_s, -\infty), (z_s, P(z_s))]$. 
We then conclude that $P$ is indeed not linear in $z_s$ and that the line $[(z_s, -\infty), (z_s, P(z_s))]$ is precisely needed to rebalance the graph $\Gamma_X(P)$, 
which proves the claim.

At first, let $z \in X\backslash\{z_s\}$ and consider $\rho \in \A^{1, \an}$ such that $\varphi_{\trop}(\rho) = z$.
Then the ultrametric triangle inequality implies $P(z)=\log (\rho(x-b))$, which precisely means $\varphi'_{\trop}(\rho) = (z, P(z))$,
and that this is the unique preimage of $z$ under $\Trop(\pi)$.

Next, we consider the preimage of the remaining point $z_s\in X$.
Let $\eta(b, t) \in \A^{1, \an}$ be given by the multiplicative seminorm $ f \mapsto \sup_{c \in D(b, t)} |f(c)|$ with $t\geq0$.
Using the ultrametric triangle inequality, we get $\varphi_{\trop} (\eta(b, t) ) = z_s$ and 
$\varphi'_{\trop}(\eta(b,t)) = (z_s, \log(t))$, for $- \infty \leq \log(t) \leq P(z_s)$.
This shows 
\begin{align}\label{Equ. Urbild z_s}
[(z_s, -\infty), (z_s, P(z_s))] \subset \Trop(\pi)^{-1}(\{z_s\}).
\end{align}
For the other inclusion, observe that by the ultrametric triangle inequality, 
we have $\rho(x - b) \leq \max (\rho (x- a_i), |a_i - b|)$ for all $i\in [r]$, which shows 
$\rho(x - b) \leq P(z_s)$ for all $\rho\in \varphi_{\trop}^{-1}(z_s)$. 
Thus, the $(r+1)$-th coordinate of any point in the fibre of $z_s$ is bounded above by $P(z_s)$. 
This shows that we have equality in (\ref{Equ. Urbild z_s}).

All together, we obtain 
\begin{align*}
X'=\Gamma_X(P) \cup [(z_s, -\infty), (z_s, P(z_s))].
\end{align*}
By Lemma \ref{Lemma P properties}, $P$ is affine everywhere except possibly at $z_s$. 
Furthermore, we know by Theorem \ref{Theorem tropical curve} and Lemma \ref{Prop weights} 
that $X'$ is a tropical curve with all weights equal to $1$. 
Thus, $P$ can not be affine in $z_s$ because otherwise $X'$ would not satisfy the balancing condition. 
Consequently, the line $[(z_s, -\infty), (z_s, P(z_s))]$  is precisely needed to rebalance the graph $\Gamma_X(P)$ 
as explained in Construction \ref{Tropicalmodification}.
\end{proof}

Theorem \ref{theorem tropical manifold} is a special instance of the fact that tropicalizations 
of linear subspaces are tropical manifolds, which was to the authors' knowledge first observed by Speyer \cite{Speyer}.
We give a self-contained proof in our case using Theorem \ref{theorem refinement modification}.

\begin{satz}\label{theorem tropical manifold}
Let $\varphi\colon \A^1 \to \A^r$ be a linear embedding. 
Then the tropical curve $\Trop_\varphi(\A^1)$  is smooth.
\end{satz}

\begin{proof}
We do induction on $r$, with $r = 1$ being obvious since $\T$ is smooth.
For the induction step let $\varphi' \colon \A^1 \to \A^{r+1}$ be given by $(x- a_1), \dots,(x-a_r),(x-b)$ and 
we need to show that $\Trop_{\varphi'}(\A^1)$ is smooth. 
Note that we already know that all weights of $\Trop_{\varphi'}(\A^1)$ are equal to $1$ by Lemma \ref{Prop weights}. 
For the other required properties in Definition \ref{Def. val, dim, glatt}, 
we consider $\varphi \colon \A^1 \to \A^r$ which is given by $(x- a_1), \dots,(x-a_r)$. 

If $b = a_i$ for some $i$, then as in the proof of Theorem \ref{theorem refinement modification}, 
$\Trop(\pi)$ is an isomorphism and we are done since $\Trop_\varphi(\A^1)$ is smooth by induction hypothesis. 
Thus we will in the following assume $b \neq a_i$ for all $i$. 
We have seen in Theorem \ref{theorem refinement modification} that 
$\Trop(\pi)\colon\Trop_{\varphi'}(\A^1)\to \Trop_\varphi(\A^1)$ is a tropical modification
and the vertices of $\Trop_{\varphi'}(\A^1)$ are precisely the preimages of the ones of $\Trop_\varphi(\A^1)$, plus $(z_s, -\infty)$ and $(z_s, P(z_s))$, 
where $z_s$ denotes the point  $(\log\vert b-a_i\vert)_{i\in [r]}$ and $P$ the function from Definition \ref{definition P}.

By induction hypothesis, we know that $\Trop_{\varphi}(\A^1)$ is smooth. 
For a vertex $z$ of $\Trop_{\varphi'}(\A^1)$ which is neither $(z_s, -\infty)$ nor $(z_s, P(z_s))$,
we have invariance of valence $\val(z) = \val(\Trop(\pi)(z))$ and dimension $\dim(z) = \dim(\Trop(\pi)(z))$.
Thus $z$ is a smooth point since $\Trop(\pi)(z)$ is.  

We examine now the situation at $z = (z_s, P(z_s))$.
Let $\sigma_1,\dots,\sigma_k$ be the edges adjacent to $z_s$.
Denote by $\sigma^P_i$ the image of $\sigma$ under the map $y \mapsto (y, P(y))$. 
The edges adjacent to $z$ are then given by $\sigma^P_1,\dots,\sigma^P_k, \{z_s\} \times [P(z_s), -\infty]$
and thus $\val(z) = \val(z_s) +1$. 
We have $\dim(z_s) = \dim \langle \nu_{z_s, \sigma_1},\dots, \nu_{z_s, \sigma_k} \rangle$ 
and $\nu_{z, \sigma^P_i} = (\nu_{z_s, \sigma_i}, c_i)$ for some $c_i \in \R$. 
Then 
\begin{align*}
\dim (z) &= \dim \langle \nu_{z, \sigma_1^P}, \dots, \nu_{z, \sigma_k^P}, (0,\dots,0,1) \rangle \\
&= \dim \langle (\nu_{z, \sigma_1}, 0), \dots, (\nu_{z, \sigma_k},0), (0,\dots,0,1) \rangle = \dim(z_s) +1.
\end{align*} 
and consequently $\val(z) = \dim(z) +1$.  

The point $(z_s, -\infty)$ lies at infinity, thus $\dim(z) = 0$, and has only the adjacent edge $[(z_s, -\infty), (z_s, P(z_s))]$, thus $\val(z) = 1$. 
Altogether, we have $\val(z) = \dim(z) +1$ for all vertices of $\Trop_{\varphi'}(\A^1)$, 
which precisely means that $\Trop_{\varphi'}(\A^1)$ is smooth, completing the induction. 
\end{proof}

\begin{kor}\label{Kor PD}
Let $\varphi\colon \A^1 \to \A^r$ be a linear embedding. Then every open subset $\Omega$ of $\Trop_\varphi(\A^1)$ has $\PD$.
\end{kor}
\begin{proof}
The assertion follows directly by Theorem \ref{theorem tropical manifold} and Theorem \ref{PD open subsets of smooth tropical curves}.
\end{proof}

\subsection{Calculating cohomology with compact support using only one tropical chart}
In this subsection, we introduce linear $\A$-tropical charts and show that for any standard open subset $V$ of $\A^{1,\an}$ 
it suffices to consider one linear $\A$-tropical chart $(V, \varphi)$ to determine the cohomology with compact support of $V$ 
(cf.~Theorem \ref{thm Cohom. lineare trop.}).

\begin{defn}
An $\A$-tropical chart $(V, \varphi)$ is called a \emph{linear $\A$-tropical chart} if the map $\varphi$ is a linear embedding.

An $\A$-tropical subchart $(V', \varphi')$ is called a \emph{linear $\A$-tropical subchart} if $\varphi'$ is a linear refinement of $\varphi$. 
\end{defn}

The next proposition shows that, when defining forms  on $\A^{1, \an}$, we may restrict our attention to linear $\A$-tropical charts.

\begin{prop}\label{Proposition 3 ins 1}
Let $V$ be an open subset of $\A^{1,\an}$ and $(V, \varphi)$ an $\A$-tropical chart.
Then there exists a linear $\A$-tropical chart $(V, \varphi')$ such that  for all $\alpha \in \AS^{p,q}(\varphi_{\trop}(V))$ there exists 
$\alpha' \in \AS^{p,q}(\varphi'_{\trop}(V))$ such that 
\begin{align*}
(V, \varphi, \alpha) = (V, \varphi', \alpha') \in \AS^{p,q}(V).
\end{align*}
\end{prop}

\begin{proof}
Let $U\subset \A^{1}$ be the domain of $\varphi$ and $f_i, g_i$ such that $\varphi\colon U \to \A^r$ is  given by $f_i / g_i$. 
Write $f_i = c_i  \prod_{j = 1} ^{s_i}(x - a_{ij})$ and $g_i = d_i  \prod_{k = 1} ^{t_i}(x - b_{ik})$. 
Let 
$\varphi'\colon \A^1 \to \prod_{i = 1}^{r} \left(\A^{s_i} \times \A^{t_i}\right)$ be
the closed embedding given by $(x- a_{ij})$ and $(x - b_{ik})$ for $i = 1,\dots,r$, $j = 1,\dots,s_i$ and $k = 1,\dots,t_i$.
Note that $\varphi'(U)$ is contained in $\prod_{i = 1}^{r} \left(\A^{s_i} \times \G_m^{t_i}\right)$
since the $g_i$ do not vanish on $U$. 
The maps
\begin{align*}
\eta_i \colon \A^{s_i} \times \G_m^{t_i} \rightarrow \A^1
\end{align*}
which are given by $x \mapsto \frac {c_i \prod_{j=1}^{s_i} T_{ij}}{d_i \prod_{k=1}^{t_i} S_{ik}}$ induce a map
\begin{align*}
\eta \colon \prod_{i = 1}^{r} \left(\A^{s_i} \times \G_m^{t_i}\right) \rightarrow \A^r. 
\end{align*}
Restricting our attention to the respective images of $U$, one can easily check on coordinate rings that the diagram
\begin{align*}
\begin{xy}
\xymatrix{
& (\varphi \times \varphi')(U) \ar[rd]^{\pi_2} \ar[ld]_{\pi_1} \\
\varphi(U) && \ar[ll]^/.9 em/{\eta} \varphi'(U) \\
& U \ar[ru]_{\varphi'} \ar[lu]^{\varphi} \ar[uu]_/.9em/{~\varphi \times \varphi'}
}
\end{xy}
\end{align*}
commutes, where $\varphi \times \varphi'\colon U \rightarrow \A^r \times \prod_{i = 1}^{r} \left(\A^{s_i} \times \A^{t_i}\right)$.

Since $\eta$ is a torus equivariant map composed with a multiplicative translation, it induces an extended affine map $\Trop(\eta)$
on the tropicalizations, which is given in the following way: 
We denote a point in $\prod_{i = 1}^{r} \left(\T^{s_i} \times \T^{t_i} \right)$ by $(y_1, z_1, \dots ,y_r,z_r)$ where 
$y_i = (y_{i,1}, \dots, y_{i,s_i}) \in \T^{s_i}$ and $z_i = (z_{i,1},\dots,z_{i,t_i}) \in \T^{t_i}$. 
Then for each $i$ we have
\begin{align*}
\Trop(\eta_i)\colon  \T^{s_i} \times \R^{t_i} &\rightarrow \T \\
 (y_{i,1}, \dots, y_{i,s_i}, z_{i,1},\dots,z_{i,t_i}) &\mapsto \sum y_{i,j} - \sum z_{i,k} + \log (c_i /d_i) .
\end{align*}
and 
\begin{align*}
\Trop(\eta)\colon \prod_{i = 1}^{r} \left(\T^{s_i} \times \R^{t_i} \right) &\rightarrow \T^r \\
(y_1, z_1, \dots ,y_r,z_r) &\mapsto (\Trop(\eta_i)(y_i, z_i))_{i \in [r]}. 
\end{align*}
We obtain the following commutative diagram of tropicalizations:
\begin{align*}
\begin{xy}
\xymatrix{
& \Trop_{\varphi \times \varphi'}(U) \ar[ld]_{\Trop(\pi_{1})} \ar[rd]^{\Trop(\pi_{2})}  \\
\Trop_{\varphi}(U)   && \ar[ll]^{\Trop(\eta)} \Trop_{\varphi'}(U)
}
\end{xy}
\end{align*}
For $\alpha \in \AS^{p,q}(\varphi_{\trop}(V))$ we define $\alpha' := \Trop(\eta)^*\alpha \in \AS^{p,q}(\varphi'_{\trop}(V))$. 
Note that $V = \varphi'^{-1}_{\trop}(\Trop(\eta)^{-1}(\varphi_{\trop}(V))$, thus  $(V, \varphi')$ is indeed a tropical chart.
Now the discussion after Definition \ref{defn:GATcharts} shows that $(V, \varphi \times \varphi')$ is a
common subchart of $(V, \varphi)$ and $(V, \varphi')$. 
The commutativity of the last diagram shows that $\Trop(\pi_1)^* \alpha = \Trop(\pi_2)^* \alpha'$ which precisely means 
$(V, \varphi, \alpha) = (V, \varphi', \alpha') \in \AS^{p,q}(V)$.
\end{proof}

\begin{lem} \label{lemma existence chart}
Let $V$ be a standard open subset of $\A^{1, \an}$. 
Then $V$ admits a linear $\A$-tropical chart $(V, \varphi)$. 
\end{lem}
\begin{proof} 
By definition
\begin{align*}
V = \{ x \in \A^{1,\an} \mid b_i < |f_i(x)| < c_i, i = 1,\dots,r\}
\end{align*}
for polynomials $f_1,\dots,f_r \in K[x]$ and elements $b_i \in \R$ and $c_i \in \R_{>0}$. 
We take $g_1,\dots,g_s$ such that $f_1,\dots,f_r,g_1,\dots,g_s$ generate $K[x]$ as a $K$-algebra
and denote by $\varphi$ the corresponding closed embedding. 
Then $V$ is precisely the preimage under $\varphi_{\trop} \colon \A^{1,\an}\to \T^{r+s}$ of the product of intervals of the form 
$[ - \infty, \log(c_i) )$, $(\log(b_{i}), \log(c_{i}))$ and $[-\infty, \infty)$ in $\T^{r+s}$. 
Thus $(V, \varphi)$ is an $\A$-tropical chart, and so the claim follows by Proposition \ref{Proposition 3 ins 1}.
\end{proof}

\begin{lem}\label{Lemma onechart}
Let $V$ be an open subset of $\A^{1,\an}$ which admits an $\A$-tropical chart $( V, \varphi)$. 
For every $\alpha \in \AS^{p,q}_c(V)$ there exists an $\A$-tropical chart $(V,\Phi)$ with $\Phi$ defined on all of $\A^1$ such that
$\alpha$ is the pullback of a form $\alpha' \in \AS^{p,q}_c(\Phi_{\trop}(V))$. 
\end{lem}
\begin{proof}
Let $\alpha \in \AS^{p,q}_c(V)$ be given by a family $(V_i, \varphi_i, \alpha_i)_{i \in I}$. 
We fix a finite subset $I'$ of $I$ such that $\supp(\alpha) \subset V' := \bigcup_{i \in I'} V_i$. 
Since linear $\A$-tropical charts in particular are defined on all of $\A^1$, 
we may assume by Proposition \ref{Proposition 3 ins 1} that each of the $\varphi_i$ is defined on all of $\A^1$. 
Denote $\Phi := \varphi \times \prod_{i \in I'} \varphi_i$, which is a refinement of all $\varphi_i$ with $i \in I'$ and $\varphi$. 
Thus $(V, \Phi)$ and $(V_i, \Phi)$ for $i \in I'$ are $\A$-tropical charts 
by the discussion after Definition \ref{defn:GATcharts} and consequently $(V', \Phi)$ is. 
We denote by $\alpha'_i$ the pullback of $\alpha_i$ to $\Phi_{\trop}(V_i)$.  
Then $\alpha'_i|_{\Phi_{\trop}(V_j)} - \alpha'_j \vert_{\Phi_{\trop}(V_i)} = 0$ since 
$\Phi_{\trop}^*(\alpha'_i|_{\Phi_{\trop}(V_j)} - \alpha'_j \vert_{\Phi_{\trop}(V_i)}) = \alpha|_{V_i \cap V_j} - \alpha|_{V_j \cap V_i} = 0$
and $\Phi_{\trop}^*$ is injective by Proposition \ref{pullback injective}.
Thus the forms $(\alpha'_i)_{i \in I'}$ glue to
a form $\alpha' \in \AS^{p,q}(\Phi_{\trop}(V'))$ which pulls back to $\alpha|_{V'}$. 
By Proposition \ref{pullback injective}, we have $\supp(\alpha') = \Phi_{\trop}(\supp(\alpha|_{V'}))$, thus it is compact.  
Then extending $\alpha'$ by zero to a form on $\Phi_{\trop}(V)$ shows that $\alpha$ can be defined 
by one triple $(V, \Phi, \alpha')$. 
\end{proof}

\begin{satz} \label{thm Cohom. lineare trop.}
Let $V$ be a standard open subset of $\A^{1, \an}$. 
Then 
\begin{align} 
\label{limit forms} \AS^{p,q}_c(V) &= \varinjlim \AS^{p,q}_c(\varphi_{\trop}(V)) \text{ and } \\
\label {limit cohomology} \HH^{p,q}_c(V) &= \varinjlim \HH^{p,q}_c(\varphi_{\trop}(V)),
\end{align}
where the limits run over the linear $\A$-tropical charts $(V,\varphi)$.

Further, for a linear $\A$-tropical chart $(V, \varphi)$ we have that
\begin{align}
\label{iso cohomology} \HH^{p,q}_c(\varphi_{\trop}(V)) \rightarrow \HH^{p,q}_c(V)
\end{align}
is an isomorphism. 
\end{satz}

\begin{proof}
For any $\A$-tropical chart $(V,\varphi)$, the pullback along the proper map $\varphi_{\trop}$ induces a well defined morphism 
$\AS^{p,q}_c(\varphi_{\trop}(V))\to\AS^{p,q}_c(V)$.
By Definition \ref{Def. forms}, this map is compatible with pullback between charts. 
Thus the universal property of the direct limit leads to a morphism $\Psi\colon\varinjlim \AS^{p,q}_c(\varphi_{\trop}(V))\to \AS^{p,q}_c(V)$,
where the limit runs over all linear $\A$-tropical charts of $V$. 
By Lemma \ref{lemma existence chart}, there exists an  $\A$-tropical chart for $V$. 
Further, by Lemma \ref{Lemma onechart}, every $\alpha \in \AS^{p,q}_c(V)$ can be defined by one chart $(V, \varphi)$ 
which we may assume to be a linear $\A$-tropical chart by Proposition \ref{Proposition 3 ins 1}. 
Hence, $\Psi$ is surjective. Since the pullback along $\varphi_{\trop}$ is injective by Proposition \ref{pullback injective} and 
$\alpha \in \AS^{p,q}_c(V)$ can be defined by only one chart, the morphism $\Psi$ is injective as well.
This shows (\ref{limit forms}). 
Equation (\ref{limit cohomology}) follows because direct limits commute with cohomology. 

By Theorem \ref{theorem refinement modification}, 
in (\ref{limit cohomology}) all transition maps are pullbacks along compositions of tropical modifications,
thus isomorphisms by  Proposition \ref{prop:closedmodification}. 
This shows (\ref{iso cohomology}). 
\end{proof}

\section{Poincar\'e duality} \label{Section 4}

The goal of this section is to prove Poincar\'e duality for a class of open subsets of $\Xan$ for a smooth algebraic curve $X$.  
To do this, we will first prove a lemma which lets us deduce Poincar\'e duality from local considerations.
Poincar\'e duality is our key statement to prove Theorem \ref{Theorem Mumford curve Kohom.} and 
Theorem \ref{Theorem Teilmenge Koho.}, where we calculate the dimension of the cohomology.

\begin{bem} \label{bem invariant}
From the definition of $\AS^{p,q}$ given in this paper (cf.~Definition \ref{Def. forms}), it is not clear 
that this definition (and the associated $d''$, the wedge product and integration) 
are functorial along analytic morphisms which do not come from algebraic morphisms.

We will briefly explain why this is the case, at least when $K$ is non-trivially valued. 
In that case, there is the definition of $\AS^{p,q}$ by Gubler \cite{Gubler}, 
which is equivalent to the one presented here by \cite[Theorem 3.2.41 \& Lemma 3.2.59]{JellThesis}. 
There is also the definition of Chambert--Loir and Ducros, which is purely analytic, thus in particular functorial along analytic morphisms. 
This approach is also equivalent to the one by Gubler, as is shown in \cite[Proposition 7.2 \& Proposition 7.11]{Gubler}.

In total, we obtain that if $K$ is non-trivially valued, our construction of $\AS^{p,q}$, $d''$, $\vedge$ and integration is equivalent 
to the ones by Chambert--Loir and Ducros, thus in particular functorial along analytic morphisms. 
This implies that $\PD$ is also functorial since this is just a combination of the wedge product and integration.
\end{bem}

Recall the definitions of $\LS^p$ and $\GS^p$ from Definition \ref{Defn nach forms}. 
Note that $\LS^p(V) = \HH^{p,0}(V)$ and $\GS^p(V) = \HH^{1-p,1}_c(V)^*$ for an open subset $V\subset \Xan$.  
We will use many times that $\HH^{q}({\AS^{1-p,1- \bullet}_c}^*(V), d''^*) = \HH^{1-p,1-q}_c(V)^*$ (cf.~Lemma \ref{Lem compact support sheaf}).
We start with the following general observation, which allows us to prove $\PD$ using local considerations. 

\begin{lem} \label{Lem PD local}
Let $X$ be an algebraic curve and $V \subset \Xan$ an open subset. 
Assume that $\PD\colon \LS^p_V \rightarrow \GS^{1-p}_V$ is an isomorphism of sheaves on $V$ and 
that further the complex $({\AS^{1-p, 1-\bullet}_{V, c}}^*,d''^*)$ is exact in positive degree.
Then $V$ has $\PD$.
\end{lem}
\begin{proof}
In the given situation, $\PD\colon \AS^{p,\bullet}_{V} \rightarrow {\AS^{1-p,1-\bullet}_{V, c}}^*$ 
is a quasi-isomorphism of complexes of sheaves on $V$. 
This is the case because the complexes are both exact in positive degree 
(by assumption resp.~Theorem \ref{Identification with singular cohomology}) and the map restricts 
to an isomorphism on the zeroth cohomology (also by assumption).
Thus the left diagram is a commutative diagram of acylic resolutions 
(note that ${\AS^{1-p,1-\bullet}_{V, c}}^*$ is flasque by Lemma \ref{Lem compact support sheaf}),
and $\PD\colon \LS^p_V \rightarrow \GS^{1-p}_V$ is an isomorphism.
\begin{align*}
\begin{xy}
\xymatrix{
&\LS^p_V \ar[r] \ar[d]^{\PD} & (\AS^{p,\bullet}_V, d'') \ar[d]^{\PD}   &  \HH^q(V, \LS^p_V)\ar@{=}[r] \ar[d]^{\PD}  & \ar[d]^\PD \HH^q(\AS^{p, \bullet}(V), d'')  \\
&\GS^{1-p}_V \ar[r]& ({\AS^{1-p,1-\bullet}_{V,c}}^*, {d''}^*) &  \HH^{q}(V, \GS^{1-p}_V)\ar@{=}[r] & \HH^{q}(\AS^{1-p,1-\bullet}_c(V)^*, {d''}^*)
}
\end{xy}
\end{align*}

We thus obtain the diagram on the right since we can use the acyclic resolutions on the left to calculate the sheaf cohomology of $\LS^p_V$ resp.~$\GS_V^{1-p}$. 
Due to $\HH^q(\AS^{p, \bullet}(V), d'') = \HH^{p,q}(V)$ and $\HH^{q}(\AS^{1-p, 1-\bullet}_c(V)^*, {d''}^*) = \HH^{1-p,1-q}_c(V)^*$, the result follows.
\end{proof}

\begin{satz} \label{thm PD}
Let $V$ be an open subset of $\mathbb{P}^{1,\an}$. 
Then $\PD\colon \LS^p_V \rightarrow \GS^{1-p}_V$ is an isomorphism of sheaves on $V$ and 
the complex $({\AS^{1-p, 1-\bullet}_{V, c}}^*,d''^*)$ is exact in positive degree.
In particular, $V$ has $PD$.
\end{satz}
\begin{proof}
That $\PD\colon \LS^p_V \rightarrow \GS^{1-p}_V$ is an isomorphism of sheaves on $V$ and 
that the complex $({\AS^{1-p, 1-\bullet}_{V, c}}^*,d''^*)$ is exact in positive degree are local conditions, 
thus we may assume $V \subset \A^{1, \an}$.
We can cover $V$ by linear $\A$-tropical charts $(W, \varphi_W)$ contained in $V$ by Lemma \ref{lemma existence chart}, 
and we can choose the $W$ sufficiently small such that
$\Omega_W := \varphi_{W, \trop}(W)$, which is an open subset of the tropical curve $\Trop_{\varphi_W}(\A^{1})$, 
is connected and 
has at most one vertex for some polyhedral structure $\CS$ on $\Trop_{\varphi_W}(\A^{1})$.
By Corollary \ref{Kor PD}, $\Omega_W$ has $\PD$.
Thus $\HH^{1-p,1-q}_c(\Omega_W)$ vanishes if and only if $\HH^{p,q}(\Omega_W)$ vanishes.
Since $\Omega_W$ has at most one vertex, by the Poincar\'e Lemma \ref{Thm Poincare Lemma one vertex} these groups vanish for $q > 0$.
Now by Theorem \ref{thm Cohom. lineare trop.}, we have an isomorphism $\HH^{1-p,1-q}_c(W) \simeq \HH^{1-p,1-q}_c(\Omega_W)$, which shows 
that $\HH^{1-p,1-q}_c(W)^*$ vanishes for $q > 0$. 
This proves exactness of ${\AS^{1-p, 1- \bullet}_{V, c}}^*$ in positive degree.

We further have the following maps 
\begin{align*}
\LS^p(\Omega_W) \inj \LS^p(W) \inj \GS^{1-p}(W) \simeq \GS^{1-p}(\Omega_W) \simeq \LS^{p}(\Omega_W).
\end{align*} 
Here, the first map is the pullback along $\varphi_{W, \trop}$, which is injective by Proposition \ref{pullback injective}. 
The second map is the $\PD$ map on $W$, which is injective by Lemma \ref{lemma PD injective}. 
The third map is the dual of the  pullback of $\varphi_{W, \trop}$ in cohomology with compact support, 
which is an isomorphism by Theorem \ref{thm Cohom. lineare trop.}. 
The fourth map is the inverse of the $\PD$ map on  $\Omega_W$, which is an isomorphism by Corollary \ref{Kor PD}. 
Since $\PD$ commutes with pullbacks, the composition is indeed the identity. 
In particular, $\PD\colon \LS^p \simeq \GS^{1-p}$ is an isomorphism. 

That $V$ has $\PD$ now follows from Lemma \ref{Lem PD local}. 
\end{proof}

In the next corollary, we need $K$ to be non-trivially valued to ensure functoriality along 
analytic maps.
\begin{kor} \label{korPD}
Assume that $K$ is non-trivially valued. 
Let $X$ be a smooth curve and $V \subset \Xan$ an open subset such that all points of type $2$ in $V$ have genus $0$. 
Then $V$ has $\PD$. 
In particular, if $X$ is a Mumford curve, every open subset of $\Xan$ has $\PD$.
\end{kor}
\begin{proof}
We claim that $\PD\colon \LS^p_V \rightarrow \GS^{1-p}_V$ is an isomorphism of sheaves on $V$ and 
that the complex $({\AS^{1-p, 1-\bullet}_{V, c}}^*,d''^*)$ is exact in positive degree.

Because this is a purely local question and $V$ is locally isomorphic to $\mathbb{P}^{1, \an}$ by Proposition \ref{Prop Thm Mumford curve}, 
we may assume that $V$ is isomorphic to an open subset of $\mathbb{P}^{1,\an}$.
Since $\AS^{\bullet, \bullet}$, $d''$ and $\PD$ are functorial along analytic maps by Remark \ref{bem invariant}, 
we may thus assume $V \subset \mathbb{P}^{1, \an}$. 
Now the claim follows from  Theorem \ref{thm PD}.

That $V$ has $\PD$ now follows from Lemma \ref{Lem PD local}. 
For a Mumford curve $X$, the analytification $\Xan$ contains no type $2$ points of positive genus 
by Theorem \ref{Satz Mumford curve} and Proposition \ref{Prop Thm Mumford curve}. 
Thus the statement for open subsets of Mumford curves follows.
\end{proof}

\begin{kor}\label{cor thm 3.13 ohne kompakt}
Let $V$ be a standard open subset of $\A^{1, \an}$
and $(V, \varphi)$ a linear $\A$-tropical chart.
Then we have  
$\HH^{p,q}(\varphi_{\trop}(V))\simeq \HH^{p,q}(V).$
\end{kor}
\begin{proof}
Since the wedge product and the 
integration map are compatible with pullback along $\varphi_{\trop}$, we find the following commutative diagram
\begin{align*}
\begin{xy}
\xymatrix{
\HH^{p,q}(\varphi_{\trop}(V)) \ar[rr]^{\PD} \ar[d] && \HH^{1-p,1-q}_c(\varphi_{\trop}(V))^* \\
\HH^{p,q}(V) \ar[rr]^{\PD} && \HH^{1-p,1-q}_c(V)^*. \ar[u]
}
\end{xy}
\end{align*}
Both $\PD$ maps are isomorphisms by Corollary \ref{Kor PD} and Theorem \ref{thm PD} and the right vertical map 
is an isomorphism by Theorem \ref{thm Cohom. lineare trop.}. 
Thus the left vertical map is one as well.
\end{proof}

\section{Cohomology of Mumford curves}\label{Section 5}

In this section, we give a calculation of $h^{p,q}(\Xan)$ for $\mathbb{P}^1$ and Mumford curves 
(cf.~Theorem \ref{Theorem Mumford curve Kohom.}). 
For these curves we can further determine $h^{p,q}$ and $h^{p,q}_c$ on a basis of the topology of $\Xan$.
Note that the Poincar\'e lemma proved in \cite{Jell} does not give a basis of open subsets $V$ such that $\HH^{p,q}(V)=0$ for $q>0$.
We will show that we obtain such a basis for open subsets of $\mathbb{P}^1$ and Mumford curves.

The key statement in this calculation is Poincar\'e duality for certain open subsets $V$ of $\Xan$ 
for smooth algebraic curves $X$ from Section \ref{Section 4}.

\begin{satz}\label{Theorem Mumford curve Kohom.}
Let either $X$ be $\mathbb{P}^1_K$ or $K$ be non-trivially valued and $X$ a Mumford curve over $K$.
We denote by $g$ the genus of $X$ and let  $p,q\in\{0,1\}$. 
Then we have
\begin{align*}
h^{p,q}(\Xan) = 
\begin{cases} 
1 \text{ if p = q}, \\
g \text{ else}.
\end{cases}
\end{align*}
\end{satz}
\begin{proof}
We have $h^{0,0}(\Xan) = 1$  by Theorem \ref{Identification with singular cohomology} and 
$h^{0,1}(\Xan) = g$ by Theorem \ref{Satz Mumford curve}. 
Thus  $h^{1,1}(\Xan) = 1$ and $h^{1,0}(\Xan) = g$ follow from Theorem \ref{thm PD} if $X=\mathbb{P}^1_K$, and from Corollary \ref{korPD} 
if $X$ is a Mumford curve over a non-trivially valued field.  
\end{proof}

We now use the results on the structure of non-archimedean curves from \cite[Section 3 \& 4]{BPR2} and work with semistable vertex sets, 
skeleta and the corresponding retraction map defined there.
Since these results are only worked out for non-trivially valued $K$, we assume from now on that $K$ is non-trivially valued.
For a semistable vertex set $V_0$ and a smooth projective curve $X$, we write $\Sigma (X, V_0)$ for the skeleton and 
$\tau=\tau_{V_0}$ for the retraction from $\Xan$ to $\Sigma (X, V_0)$.

The following definition is inspired by \cite[Corollary 4.27 \& Definition 4.28]{BPR2}.

\begin{defn}\label{Def. simple}
Let $X$ be a smooth curve. 
An open subset $V$ of $\Xan$ is called \emph{simple} if it is either isomorphic to an open disc
or an open annulus or if there exists a semistable vertex set $V_0$ and a simply connected open subset $W$ of $\Sigma(X, V_0)$
 such that 
$V= \tau_{V_0}^{-1}(W)$. 
A simple open subset $V$ is called \emph{strictly simple}
if its closure in $\Xan$ is simply connected in the case of a disc or an annulus or 
if the closure of $W$ in $\Sigma(X, V_0)$ is simply connected for $V=\tau^{-1}(W)$.
\end{defn}

\begin{prop}
Strictly simple open subsets form a basis of the topology of $\Xan$. 
\end{prop}
\begin{proof}
For simple open subsets, this follows from \cite[Corollary 4.27]{BPR2}.
Thus it is enough to show that we can cover simple subsets by strictly simple ones.
If $V$ is an open disc or annulus, then we can cover $V$ by open discs and annuli whose closures are contained in $V$, 
thus their closures are simply connected again. 
If $V = \tau ^{-1}(W)$ for $W \subset \Sigma(X, V_0)$ open, then we can cover $W$ by open subsets $W_i$ 
whose closure (in $\Sigma(X, V_0)$) is contained in $W$. 
Thus the closure of $W_i$ is simply connected, and so every $V_i$ is strictly simple.
\end{proof}

For an algebraic curve $X$ and an open subset $V$ of $X^{\an}$, we denote by $\del V$ the topological boundary of $V$ inside $\Xan$.
Note that this is the same as the set of limit points of sequences in $V$ which are not contained in $V$ by \cite[Corollaire 5.5]{Poineau}.
For an open subset $W$ of a skeleton $\Sigma(X, V_0)$ we denote by $\del W$ its boundary inside $\Sigma(X, V_0)$. 
Again this is the set of limit points of sequences in $W$ which are not contained in $W$ since $\Sigma(X, V_0)$ is a finite graph, 
thus satisfies the first countability axiom.

\begin{lem} \label{the boundaries agree}
Let $X$ be a smooth projective curve, $V_0$ a semistable vertex set, $W \subset \Sigma(X,V_0)$ an open subset and $V = \tau_{V_0}^{-1}(W)$. 
Then $\del V = \del W$.
\end{lem}

\begin{proof}
Since $W \subset V$, we have that $\del W$ is contained in $\overline{V}$. 
If there is a point $x\in \del W\cap V$, we have $x=\tau(x)\in W$, where we use that $\tau$ is the identity on the skeleton. 
This contradicts $x\in \del W$, and so  $\del W\subset\del V$. 

Suppose now that $x \in \del V$. 
Let $(x_n)$ be a sequence in $V$ converging to $x$. 
If there are infinitely many $x_n$ in $W$, then this subsequence also converges to $x$ and thus $x \in \del W$. 
Otherwise, we may assume that $x_n \in V \setminus W$ for all $n$. 
By the definition of $\tau$ and \cite[Lemma 3.4 (3)]{BPR2}, we have 
\begin{align*}
V \setminus W = \coprod \limits_{y \in W} \tau^{-1}(\{y\}) \setminus \{y\}.
\end{align*}
Since the sequence $(x_n)$ converges, there need to be infinitely many $x_n$ in one of its connected components. 
Passing again to a subsequence, we may thus assume that there exists $y \in W$ such that $x_n \in \tau^{-1}(\{y\}) \setminus \{y\}$ for all $n$. 
Thus $x \in \overline{\tau^{-1}(\{y\}) \setminus \{y\}} = \tau^{-1}(\{y\})\subset V$ by \cite[Lemma 3.2]{BPR2}, which contradicts $x \in \del V$. 
\end{proof}

\begin{kor} \label{the boundary is finite}
Let $X$ be a smooth projective curve and $V$ a strictly simple open subset of $\Xan$. 
Then $\del V$ is finite.
\end{kor}
\begin{proof}
If $V$ is an open disc resp.~an open annulus with simply connected closure, it has one resp.~two boundary points 
\cite[Lemma 3.3 \& Lemma 3.6]{ABBR}.
Otherwise this is a direct consequence of Lemma \ref{the boundaries agree} and the fact that $\del W$ (for $W$ as defined there)
is finite. The latter is the case because $W$ is a connected open subset of a finite graph. 
\end{proof}

\begin{lem}\label{Lemma Sternchen}
Let $X$ be a smooth projective curve and $V \subset \Xan$ a strictly simple 
open subset such that $\# \del V = k$.
Then $V$ is properly homotopy equivalent to the one point union of $k$ copies of half open intervals, glued at the closed ends.
\end{lem} 

\begin{proof}
At first, we consider the case where $V$ is an open disc of radius $r$ or an open annulus with radii $r_1 < r_2$. 
Note that if $V$ is a disc, then $k = 1$. 
If $V$ is an annulus, then $k = 2$ since the closure of $V$ is simply connected.
In both situations, we may thus assume that $X=\mathbb{P}^1$. 
 For any $R>0$ we denote by $\zeta_{R}$ the point in $\mathbb{P}^{1,\an}$ which is given by the seminorm $ f \mapsto \sup_{c \in D(0,R)} |f(c)|$
 and by $[\zeta_{R}, \zeta_{R'}]$ the unique path between two such points in the uniquely path connected space 
 $\mathbb{P}^{1,\an}$ \cite[Lemma 2.10]{BR}.

If $V$ is an open disc of radius $r$, by change of coordinates we may assume that $V$ is the connected component of 
$\mathbb{P}^{1,\an}\backslash \{\zeta_{r}\}$ containing $0$.   
We take $V_0$ to be the semistable vertex set consisting of $\zeta_{R_1}$ and $\zeta_{R_2}$ with $R_1,R_2\in \vert K^\times\vert$ and $R_1<r<R_2$.
Then $\Sigma(\mathbb{P}^1, V_0) = [\zeta_{R_1},\zeta_{R_2}]$ and we find $V = \tau^{-1}([\zeta_{R_1}, \zeta_{r}))$.
Since $\tau$ is a homotopy equivalence between compact spaces, it induces a proper homotopy equivalence $V \rightarrow [\zeta_{R_1}, \zeta_{r})$.

If $V$ is an annulus with radii $r_1 < r_2$, we take $R_1,R_2\in \vert K^\times\vert$ with $R_1 < r_1$ and $r_2 < R_2$.   
Then $\Sigma(\mathbb{P}^1, V_0) = [\zeta_{R_1},\zeta_{R_2}]$ and now $V = \tau^{-1}((\zeta_{r_1},\zeta_{r_2}))$. 
Again, $(\zeta_{r_1},\zeta_{r_2})$ is properly homotopy equivalent to $V$.

Now let $V = \tau^{-1} (W)$ for a simply connected open  subset $W$ of $\Sigma(X, V_0)$ for some semistable vertex set $V_0$. 
Again, $W$ is properly homotopy equivalent to $V$. 
Since $\overline{W}$ is simply connected, $W$ is the interior of a simply connected finite graph contained in $\Sigma(X, V_0)$, 
thus properly homotopy equivalent to the one point union of $\# \del W$ copies of half open intervals, glued at the closed ends. 
The result now follows from Lemma \ref{the boundaries agree}.
\end{proof}

Note that the following theorem applies to all strictly simple open subsets $V$ of $\Xan$ if $X$ is a Mumford curve or $\mathbb{P}^1$
by Theorem \ref{Satz Mumford curve} and Proposition \ref{Prop Thm Mumford curve}.

\begin{satz} \label{Theorem Teilmenge Koho.}
Let $X$ be a smooth projective  curve over $K$ and $p,q\in\{0,1\}$. 
Let $V$ be a strictly simple open subset of $\Xan$ such that all type $2$ points in $V$ have genus $0$ and
denote by $k := \# \del V$ the finite number of boundary points. 
Then we have 
	\begin{align*}
	h^{p,q}(V) = 
	\begin{cases}
	1 &\text{ if } (p,q) = (0,0) \\
	k - 1 &\text{ if } (p,q) = (1,0) \\
	0 &\text{ if } q \neq 0
	\end{cases}
	\text{ and }
	h^{p,q}_c(V) = 
	\begin{cases}
	1 &\text{ if } (p,q) = (1,1) \\
	k - 1 &\text{ if } (p,q) = (0,1) \\
	0 &\text{ if } q \neq 1.
	\end{cases}
	\end{align*}
\end{satz}

\begin{proof}
First, note that $V$ has $\PD$ by Corollary \ref{korPD}.
Thus it is sufficient to calculate $h^{0,q}(V)$ and $h^{0,q}_c(V)$. 
By identification with singular cohomology (cf.~Theorem \ref{Identification with singular cohomology}), 
we only have to calculate $h^{q}_{\sing}(V)$ and $h^{q}_{c, \sing}(V)$,
which are invariant under proper homotopy equivalences. 
Thus by Lemma \ref{Lemma Sternchen}, we have to calculate $h^{q}_{\sing}(Y)$ and $h^{q}_{c, \sing}(Y)$ for $Y$  a one point union of $k$ intervals. 
Since $Y$ is connected and contractible, we have  $h^{0}_{\sing}(Y) = 1$ and $h^{1}_{\sing}(Y) = 0$.
Calculating its cohomology with compact support is an exercise in algebraic topology, which we lay out for the convenience of the reader.
We have 
\begin{align*}
\HH^{q}_{c, \sing}(Y) = \varinjlim \limits_{E \subset Y \text{ compact}} \HH^q_{\sing}(Y, Y \setminus E)
\end{align*}
by \cite[p. 244]{Hatcher}. 
Each compact subset $E'$ of $Y$ is contained in a connected compact subset $E$ where $E$ intersects all intervals from which $Y$ is glued. 
Thus we may restrict our attention to those $E$.
For every pair $E_1 \subset E_2$ of such subsets, $(Y, Y \setminus E_2) \inj (Y, Y \setminus E_1)$ is a homotopy equivalence, 
thus all transition maps in the limit are isomorphisms. 
We have $\HH^{0}_{\sing}(Y) = \R$, $\HH^0_{\sing}(Y \setminus E) = \R^k$ and for both $Y$ and $Y \setminus E$ higher cohomology groups vanish. 
Now using the long exact sequence of pairs gives the desired result. 
\end{proof}

\bibliographystyle{alpha}

\begin{thebibliography}{BIMS15}

\bibitem[Ber90]{BerkovichSpectral}
Vladimir~G. Berkovich.
\newblock {\em Spectral theory and analytic geometry over non-{A}rchimedean
  fields}, volume~33 of {\em Mathematical Surveys and Monographs}.
\newblock American Mathematical Society, Providence, RI, 1990.

\bibitem[Ber07]{BerkoIntegration}
Vladimir~G. Berkovich.
\newblock {\em Integration of one-forms on {$p$}-adic analytic spaces}, volume
  162 of {\em Annals of Mathematics Studies}.
\newblock Princeton University Press, Princeton, NJ, 2007.

\bibitem[BIMS15]{BIMS15}
Erwan Brugall\'e, Ilia Itenberg, Grigory Mikhalkin, and Kristin Shaw.
\newblock Brief introduction to tropical geometry.
\newblock In {\em G\"okova Geometry-Topology conference}, 2015.

\bibitem[BPR14]{BPR2}
Matt Baker, Sam Payne, and Joseph Rabinoff.
\newblock On the structure of nonarchimedean analytic curves.
\newblock \href{https://arxiv.org/pdf/1404.0279v1.pdf}{\tt arXiv:1404.0279 },
  2014.

\bibitem[BR10]{BR}
Matthew Baker and Robert Rumely.
\newblock {\em Potential theory and dynamics on the {B}erkovich projective
  line}, volume 159 of {\em Mathematical Surveys and Monographs}.
\newblock American Mathematical Society, Providence, RI, 2010.

\bibitem[Bre97]{Bredon}
Glen~E. Bredon.
\newblock {\em Sheaf theory}, volume 170 of {\em Graduate Texts in
  Mathematics}.
\newblock Springer-Verlag, New York, second edition, 1997.

\bibitem[CLD12]{CLD}
Antoine Chambert-Loir and Antoine Ducros.
\newblock Formes diff\'{e}rentielles r\'{e}elles et courants sur les espaces de
  {B}erkovich.
\newblock 2012.
\newblock \url{http://arxiv.org/abs/1204.6277}.

\bibitem[Gub13]{Gubler2}
Walter Gubler.
\newblock A guide to tropicalizations.
\newblock In {\em Algebraic and combinatorial aspects of tropical geometry},
  volume 589 of {\em Contemp. Math.}, pages 125--189. Amer. Math. Soc.,
  Providence, RI, 2013.

\bibitem[Gub16]{Gubler}
Walter Gubler.
\newblock Forms and currents on the analytification of an algebraic variety
  (after {C}hambert-{L}oir and {D}ucros).
\newblock In Matthew Baker and Sam Payne, editors, {\em Nonarchimedean and
  Tropical Geometry}, Simons Symposia, pages 1--30, Switzerland, 2016.
  Springer.

\bibitem[Hat02]{Hatcher}
Allen Hatcher.
\newblock {\em Algebraic topology}.
\newblock Cambridge University Press, Cambridge, 2002.

\bibitem[Ive86]{Iversen}
Birger Iversen.
\newblock {\em Cohomology of sheaves}.
\newblock Universitext. Springer-Verlag, Berlin, 1986.

\bibitem[Jel16a]{JellThesis}
Philipp Jell.
\newblock Differential forms on {B}erkovich analytic spaces and their
  cohomology.
\newblock PhD Thesis, available at
  \url{http://epub.uni-regensburg.de/34788/1/ThesisJell.pdf}, 2016.

\bibitem[Jel16b]{Jell}
Philipp Jell.
\newblock A {P}oincar\'e lemma for real-valued differential forms on
  {B}erkovich spaces.
\newblock {\em Math. Z.}, 282(3-4):1149--1167, 2016.

\bibitem[JSS15]{JSS}
Philipp Jell, Kristin Shaw, and Jascha Smacka.
\newblock Superforms, tropical cohomology and {P}oincar\'e duality.
\newblock 2015.
\newblock \url{http://arxiv.org/abs/1512.07409}.

\bibitem[Lag12]{Lagerberg}
Aron Lagerberg.
\newblock Super currents and tropical geometry.
\newblock {\em Math. Z.}, 270(3-4):1011--1050, 2012.

\bibitem[Liu16]{Liu}
Yifeng Liu.
\newblock Weight decomposition of de rham cohomology sheaves and tropical cycle
  classes for non-archimedean spaces.
\newblock 2016.
\newblock \url{http://www.math.northwestern.edu/~liuyf/deRham.pdf}.

\bibitem[Poi13]{Poineau}
J{\'e}r{\^o}me Poineau.
\newblock Les espaces de {B}erkovich sont ang\'eliques.
\newblock {\em Bull. Soc. Math. France}, 141(2):267--297, 2013.

\bibitem[Sha13]{Shaw:IntMat}
Kristin Shaw.
\newblock A tropical intersection product in matroidal fans.
\newblock {\em SIAM J. Discrete Math.}, 27(1):459--491, 2013.

\bibitem[Sha15]{Shaw:Surfaces}
Kristin Shaw.
\newblock Tropical surfaces.
\newblock 2015.
\newblock \url{http://arxiv.org/abs/1506.07407}.

\bibitem[Sma17]{Smacka}
Jascha Smacka.
\newblock Differential forms on tropical spaces.
\newblock PhD Thesis, available at \url{https://epub.uni-regensburg.de/36262/},
  2017.

\bibitem[Spe08]{Speyer}
David~E. Speyer.
\newblock Tropical linear spaces.
\newblock {\em SIAM J. Discrete Math.}, 22(4):1527--1558, 2008.

\bibitem[Wel80]{Wells}
Raymond~O. Wells, Jr.
\newblock {\em Differential analysis on complex manifolds}, volume~65 of {\em
  Graduate Texts in Mathematics}.
\newblock Springer-Verlag, New York-Berlin, second edition, 1980.

\end{thebibliography}
\def\cprime{$'$}

\end{document}